\documentclass[11pt, reqno]{amsart}
 \usepackage{leading}
 \usepackage{blindtext}

 \newcommand{\todo}[1]{}

\usepackage[new]{old-arrows}
 
\makeatletter
\DeclareMathAlphabet{\mathpzc}{OT1}{pzc}{m}{it}
\usepackage[margin=2.4cm]{geometry}
\usepackage{amsmath, amssymb, amsfonts, amstext, verbatim, amsthm, mathrsfs}
\usepackage[mathcal]{eucal}
\usepackage{microtype}

\DeclareSymbolFont{bbold}{U}{bbold}{m}{n}
\DeclareSymbolFontAlphabet{\mathbbold}{bbold}

\usepackage[usenames]{color}
\usepackage{aliascnt} 
\usepackage[colorlinks=true,linkcolor=blue,citecolor=red,urlcolor=blue,citebordercolor={0 0 1},urlbordercolor={0 0 1},linkbordercolor={0 0 1}]{hyperref} 
\usepackage{enumerate}
\usepackage{xspace}
\usepackage{tikz-cd}
\usepackage[nameinlink]{cleveref}

\theoremstyle{plain}

\newcommand{\refnewtheoremn}[4]{
\newaliascnt{#1}{#2}
\newtheorem{#1}[#1]{#3}
\aliascntresetthe{#1}
\expandafter\providecommand\csname #1autorefname\endcsname{#4}}

\newcommand{\refnewtheorem}[3]{\refnewtheoremn{#1}{#2}{#3}{#3}}

\def\makeCal#1{
\expandafter\newcommand\csname c#1\endcsname{\mathcal{#1}}}
\def\makeBB#1{
\expandafter\newcommand\csname b#1\endcsname{\mathbb{#1}}}
\def\makeFrak#1{
\expandafter\newcommand\csname f#1\endcsname{\mathfrak{#1}}}

\count@=0
\loop
\advance\count@ 1
\edef\y{\@Alph\count@}
\expandafter\makeCal\y
\expandafter\makeBB\y
\expandafter\makeFrak\y
\ifnum\count@<26
\repeat

\newtheorem{theorem}{Theorem}[section]

\refnewtheorem{lemma}{thm}{Lemma}
\refnewtheorem{assumptions}{thm}{Assumptions}
\refnewtheorem{examples}{thm}{Examples}
\refnewtheorem{claim}{thm}{Claim}
\refnewtheorem{cor}{thm}{Corollary}
\refnewtheorem{conj}{thm}{Conjecture}
\refnewtheorem{prop}{thm}{Proposition}
\refnewtheorem{proposition}{thm}{Proposition}
\refnewtheorem{quest}{thm}{Question}
\refnewtheorem{assumption}{thm}{Assumption}

\theoremstyle{definition}
\refnewtheorem{remark}{thm}{Remark}
\refnewtheorem{remarks}{thm}{Remarks}
\refnewtheorem{defn}{thm}{Definition}
\refnewtheorem{definition}{thm}{Definition}
\refnewtheorem{notn}{thm}{Notation}
\refnewtheorem{const}{thm}{Construction}
\refnewtheorem{example}{thm}{Example}
\refnewtheorem{fact}{thm}{Fact}
\refnewtheorem{problem}{thm}{Problem}

\theoremstyle{definition}

\newcommand {\id}{\operatorname{id}}

\newcommand{\coker}{\operatorname{coker}}

\newcommand {\Hom}{\operatorname{Hom}}

\newcommand{\isom}{\cong}

\newcommand{\ul}{Ulrich }

\numberwithin{equation}{subsection}
\usepackage{expl3}

\usepackage{xparse}
\ExplSyntaxOn 

\NewDocumentCommand\alotofdots{ D(){\ldotp} O{3}}
{\mathinner{\prg_replicate:nn{#2}{#1}}}
\ExplSyntaxOff

\begin{document}
	
\title{Ulrich bundles on cyclic coverings of projective spaces}

\author{A J Parameswaran} 

\address{Kerala School of Mathematics, Kunnamangalam PO, Kozhikode-673571, Kerala,
	India}
	
\email{param@ksom.res.in, param@math.tifr.res.in}

\author{Jagadish Pine}

\address{Indian Institute of Science Education and Research, Pune,  Main Academic, Dr Homi Bhabha Rd, Pashan, Maharashtra 411008}

\email{math.jagadishpine@gmail.com, jagadish.pine@iiserpune.ac.in}

\date{}



\begin{abstract}
We prove the existence of Ulrich bundles on cyclic coverings of $\mathbb{P}^n$ of arbitrary degree $d$. Given a relatively Ulrich bundle on a complete intersection subvariety, we construct a relatively Ulrich bundle on the ambient variety. As an application, we prove that there exists
a rank $d$ Ulrich bundle on a generic cyclic covering of $\mathbb{P}^{2}$ of degree $d$, provided that the degree $d \cdot k$ of the branch divisor is even. When $d \cdot k$ is odd, we also provide an estimation of the rank of the Ulrich bundle on a generic cyclic covering of $\mathbb{P}^2$. 
\end{abstract}

\maketitle

\section{Introduction}

\label{one}
We will work over the field of complex numbers $\bC$. Let $X \, \subseteq \, \bP^N$ be a projective variety and $\cO_X(1)$ be the restriction ${\cO_{\bP^N}(1)}_{|X}$. A vector bundle $E$ on $X$ is \ul with respect to the very ample line bundle $\cO_X(1)$ if $H^{i}(X,\,  E(-i)) =0$ for $i > 0$, $H^{j}(X,\,  E(-j-1)) =0$ for $j \,  < \text{dim}(X)$; see Proposition \ref{Eisenbuddefinition}.  
The definition of \ul bundles can be extended for an ample and globally generated polarization. The motivation of this extended definition is due to the result \cite[Proposition 5.4]{eisenbud2003resultants} 
: If there exists an \ul bundle on $X$ for a polarization $H$, then it will have an \ul bundle of much larger rank for a multiple $\alpha \cdot H$ where $\alpha \in \bZ_{>0}$. In this article, by an \ul bundle for a finite, surjective map $\pi  :\, X \longrightarrow \bP^n$ we will mean a bundle $E$ on $X$ such that $\pi_*E \, \cong \, \cO_{\bP^n}^{\oplus \text{rk}(E) \cdot \text{deg}(\pi)}$. This is equivalent to the above cohomological definition with respect to the polarization $\pi^*\cO_{\bP^n}(1)$.

  A variety admitting an \ul bundle has several favorable properties. If the projective variety $X$ supports an \ul bundle, then the Chow form of $X$ has a linear determinantal or Pfaffian representation \cite{eisenbud2003resultants}. Also, they have applications in Boij-S\"{o}derberg theory \cite[Theorem 4.2]{MR2810424}, namely, the cone of cohomology tables of vector bundles on a $d$ dimensional scheme $X$ is the same as the cone of cohomology tables of vector bundles of $\bP^d$ if and only if $X$ supports an \ul bundle. In \cite{eisenbud2003resultants}, the authors asked whether every projective variety $X$ admits an \ul bundle. If so, what is the minimum possible rank of an \ul bundle? So far, these questions have been explored for certain classes of varieties. In this article, we address these questions when the central projections $\pi:\, X \longrightarrow \bP^n$ are cyclic.\\

Let $\pi\, :\, X \, \longrightarrow \, \bP^n$ be a cyclic covering (for a discussion, refer to \S \ref{cycliccovering}) of degree $d$ branched over a smooth hypersurface of degree $d \cdot k$ for some natural number $k$. On the existential question of \ul bundle, we prove the following

\begin{theorem}[{Theorem \ref{veroneseargumenttheorem}}]
	\label{introthm1}
	There exist \ul vector bundles on every smooth cyclic covering of $\bP^n$. 
\end{theorem}

The case of degree $\, d \, =\, 2\, $ was treated in \cite[Theorem 1.2]{kumar2023ulrichbundlesdoublecovers}. The outline of the proof of Theorem \ref{introthm1} is as follows. Let $|\cO_{\bP^n}(k)|\, :\, \bP^n \,\longrightarrow\, \bP^N$ be the Veronese embedding. We construct a degree $d$ polynomial $g$ from the branch divisor using the Veronese coordinates of $\bP^N$.   We consider a certain degree $d$ cyclic covering $\tilde{\pi}\,:\, \tilde{Z} \,\longrightarrow\, \bP^N$, where $\tilde{Z}\,=\, Zero(t^d-g)\, \subseteq \,\bP^{N+1}$ is a hypersurface, and $\tilde{\pi}$ is a certain linear projection from a point. The covering $\pi$ is the restriction of $\tilde{\pi}$ to $\bP^n$.  Let $\alpha\,=\,(\alpha_1,\, \alpha_2,\, \cdots,\, \alpha_d)$ be a matrix factorization of $(t^d \,-\, g)$ of size $m$ as in Proposition \ref{existencematrix}. We define the following short exact sequence
\begin{center}
	\begin{tikzcd}[sep=scriptsize]
		0 && {\mathcal{O}_{\mathbb{P}^{N+1}}(-1)}^{\oplus m} && {\mathcal{O}_{\mathbb{P}^{N+1}}}^{\oplus m} && G_i && 0
		\arrow[from=1-1, to=1-3]
		\arrow["{\alpha_i}", from=1-3, to=1-5]
		\arrow[from=1-5, to=1-7]
		\arrow[from=1-7, to=1-9]
	\end{tikzcd}
\end{center}
where $G_i$ is the sheaf $\coker(\alpha_i)$. The proof of Theorem \ref{introthm1} ultimately hinges on proving the following proposition, which states that the sheaves $G_i$ on $\tilde{Z}$ are Ulrich.
  
\begin{proposition}[Proposition \ref{keythm}]
	\label{introthm2}
	The sheaves $\tilde{\pi}_*G_i$ are trivial vector bundles of rank $m$. 
\end{proposition}

By computing the size of a matrix factorization of $(t^d-g)$, we can estimate the rank of the \ul bundle on $X$. However, this method is not effective in reducing the rank. Our next theorem gives an alternative approach to constructing \ul bundles on cyclic coverings of $\bP^n$, which can also be used in minimizing the rank. Apart from these immediate applications, the result can also be of general interest.  
We define the notion of a \emph{relatively \ul bundle} in Definition \ref{relulbun} as a generalization of \ul bundle for an arbitrary finite map between smooth, projective varieties. This notion becomes relevant in the context of the following theorem.

 \begin{theorem}[Theorem \ref{generalextensionthm}]
	\label{introthm3}
	Let $\pi\,:\, X \,\longrightarrow \, Y$ be a cyclic covering of degree $d$ of smooth projective varieties. Then we will have $\pi_* \cO_X \,\cong \,\cO_Y \,\oplus\, L^{-1} \,\oplus\, \cdots \,\oplus \, L^{-d+1}$ for some line bundle $L \,\in \,\text{Pic}(Y)$ such that the branch divisor $B \,\in\, |L^{\otimes d}|$. Let $W \,\in \,|L|$ be a smooth divisor such that $B \, \cap \, W$ is smooth (or equivalently, the intersection is transversal). We consider the following Cartesian diagram
	\begin{center}
		\begin{tikzcd}
			{{Z}} && {{X}} \\
			{W} && {Y}
			\arrow[hook, from=2-1, to=2-3]
			\arrow["\pi'", from=1-1, to=2-1]
			\arrow["\pi"', from=1-3, to=2-3]
			\arrow[hook, from=1-1, to=1-3]
		\end{tikzcd}
	\end{center}
	Then $\pi'\,:\, Z \,\longrightarrow\, W$ is a cyclic covering of degree $d$ of smooth projective varieties. Further we will assume $H^{1}(Y,\, L^{\otimes j})\,=\,0$ for all $j \,\in \,\bZ$. 
	If $F$ is a rank $r$ \emph{relatively \ul bundle} on $Z$ with respect to $\pi'$, then we can construct a \emph{relatively \ul bundle} $E$ of rank $d \cdot r$ on $X$ with respect to $\pi$.
\end{theorem}

The case of double covering in Theorem \ref{introthm3} was proved in \cite[Theorem 4.1]{kumar2023ulrichbundlesdoublecovers}. The proof of Theorem \ref{introthm3} is done in several steps in this paper. In the first step, we modify the trivial bundle $\cO_X^{\oplus d \cdot r}$ along $F$, where modification refers to taking the kernel of a natural surjective map $\cO_X^{\oplus d \cdot r} \, \longrightarrow\, F$. In the second step, we further modify this kernel along $F^{\oplus 2}$. The main challenge
in this step is the lack of an immediate surjective map from the kernel to
$F^{\oplus 2}$, which requires us to construct a suitable surjection such that the pushforward under $\pi _{*}$ of its kernel
has a certain property. The final \emph{relatively \ul bundle} $E$ is obtained by iterating similar modifications.   
 We will discuss a few applications of Theorem \ref{introthm3}. \\
 
 The theorem reduces the problem of showing the existence of an \ul bundle for a cyclic covering $\pi\,:\, X \,\longrightarrow\,\bP^n$ to finding a relatively \ul bundle on a cyclic covering of certain complete intersection curves in $\bP^n$ (see Theorem \ref{completeintthm}). In \cite[Theorem 1.2]{sebastian2022rank}, the authors have shown that a double plane branched along a smooth generic curve $B \,\subseteq\, \bP^2$ of degree $2s$ such that $s \,\ge\, 3$ admits a rank $2$ special \ul bundle. In \cite[Theorem 1.1]{kumar2023ulrichbundlesdoublecovers}, the authors have proved that every smooth double plane admits a rank $2$ \ul bundle. Applying Theorem \ref{introthm3} in this setting, we prove the following generalization:

 \begin{theorem}[Theorem \ref{evenparitythm}]
 	\label{introthm4}
 	Let $\,\pi\,:\, X \,\longrightarrow\,\bP^2$ be a generic cyclic covering of degree $d$ such that the degree of the branch divisor $d \,\cdot\, k$ is even. Then there exists an \ul bundle $E$ on $X$ of rank $d$. 
 \end{theorem}
A sketch of the proof is as follows. By \cite[Theorem 5.1]{carlini2008complete} any degree $d$ polynomial $F$ defining a plane curve can be written as $F \,=\, P_1Q_1 \,+\, P_2Q_2$ where $\text{degree}\,(P_1)\,=\, d_1$, $\text{degree}\,(P_2)\,=\, d_2$ for $d_1 \,\le\, d_2\, <\, d$. In Proposition \ref{smoothnesstrans}, we prove that for a generic plane curve defined by $F$, we can choose $P_1$ to be smooth, and $P_1$ intersects both $P_2$ and $Q_2$ transversally. A suitable application of Theorem  \ref{introthm3} reduces the proof of  Theorem \ref{introthm4} to showing the existence of a \ul line bundle for self-covering of $\bP^1$. It would be nice to prove the existence of a rank $d$ \ul bundle on $X$ when $d \,\cdot\, k$ is odd (see Question \ref{parityquest}). At present, using methods developed in this paper, we have estimated the rank when $d \,\cdot\, k$ is odd; see Theorem \ref{dkodd}, Proposition \ref{easycase}, Corollary \ref{oddcase}, and \S \ref{difficultcase}. As a further application of our methods, we have produced alternative proofs of two results of \cite{parameswaran2021ulrich}, new examples of cyclic coverings of $\bP^2$ of degree $d$ that admit Ulrich line bundles, and a new example in \S \ref{six}.\\

The outline of the paper is as follows. In \S \ref{two} we recall the definition and existence of matrix factorizations of homogeneous polynomials. We also define cyclic coverings and mention a few properties of them. In \S \ref{three} we define \ul sheaves on a projective variety. The main result of this section is the existence of \ul bundle in \S \ref{veroneseargumenttheorem}, and the main ingredient in the proof is Proposition \ref{keythm}. In \S \ref{four} we define the notion of relatively \ul bundle for a finite, surjective covering of varieties. We prove the main theorem of this section regarding the extension of a relatively \ul bundle in Theorem \ref{generalextensionthm}. Further, we prove the existence of \ul bundle on cyclic coverings of projective spaces under the assumption of the existence of a relatively \ul bundle on certain cyclic coverings of complete intersection curves in Theorem \ref{completeintthm}. In \S \ref{five} we deal with the case of cyclic coverings of $\bP^2$. The main result of this section is regarding the rank of \ul bundle when the degree of the branch locus of the cyclic covering is even in Theorem \ref{evenparitythm}. Further, we provide an estimation of rank in the odd degree case in Theorem \ref{dkodd}. In \S \ref{six} we prove two corollaries which were earlier proven in \cite{parameswaran2021ulrich} using different methods. We have provided new examples of cyclic coverings of $\bP^2$ of degree $d$ that admit Ulrich line bundles, as well as  a new example.

\section{Preliminaries}
\label{two}
 We will briefly introduce the notion of matrix factorization following \cite{HERZOG1991187} and \cite{BFb0082014}, which will be used later to prove the existence of \ul sheaves. We will also define the notion of cyclic coverings and make a few remarks on them.
\subsection{Matrix factorization}
\begin{defn} \label{MF}
	A homogeneous polynomial $f$ of degree $d$ in $\bC[x_1\,,\, x_2\,,\, \cdots,\, x_n]$ has a (linear) matrix factorization of size $m$ if $\alpha_1\, \cdot\, \alpha_2\, \cdots \,\alpha_d \,=\, f \,\cdot \,\id_{m \,\times\, m}$, where $\alpha_i$ are matrices of order $m \,\times\, m$ with entries of linear homogeneous polynomials in $\bC[x_1\,,\, x_2\,,\, \cdots,\, x_n]$.
\end{defn}

\begin{examples}
	\begin{itemize}
		\item [(1)] Let $f \,=\, x_{1} \,\cdot\, x_{2} \,\cdots\, x_{n}$, then setting $\alpha_i \,=\,(x_{i})$ gives a matrix factorization of size $1$.
		
		\item [(2)] Let $f \,=\, t^{3} \,-\, x \,\cdot\, \, y \cdot\, z$ be in $\bC[t,\, x,\, y,\, z]$. Then $f$ is an irreducible polynomial. Let $A$ be the matrix
		$A \,=\, \begin{bmatrix}
			0 & 0 & x \\
			y & 0 & 0 \\
			0 & z & 0 
		\end{bmatrix} $ which satisfies $A^{3} \,=\, (xyz)\,\cdot\, \id_{3 \times 3}$. By setting $\alpha_1 \,=\, \alpha_2 \,=\, \alpha_3 \,=\, A$, we get a matrix factorization of the polynomial $xyz$ of size $3$. Let $\zeta$ be a primitive cubic root of unity. We set $\beta_1 \,=\, (t \,\cdot\, \id_{3 \times 3}\,-\, A),\, \beta_2 \,=\, (t \,\cdot\, \id_{3 \times 3} \,-\,\zeta A)$, and  $\beta_3= (t \,\cdot\, \id_{3 \times 3}\,-\,\zeta^2 A)$, where $t \,\cdot\, \id_{3 \times 3}$ is the scalar matrix of size $3 \times 3$ with entry $t$. Then we see $\beta_1 \,\cdot\, \beta_2 \,\cdot \,\beta_3 \,=\, f \,\cdot\, \id$.
	\end{itemize}
\end{examples}

Let $f$ be a degree $d$ homogeneous polynomial. The existence of a matrix factorization of finite size for $f$ is proved in the following
\begin{proposition}\cite[Lemma 1.5]{HERZOG1991187}
	\label{existencematrix}
	Let $g \,=\, \sum_{i=1}^{s}\,\prod_{j=1}^{d}\, (a_{ij} \,\cdot\, x_{ij})$ be a homogeneous polynomial of degree $d$, then there is a matrix factorization ${\beta} \,=\, ({\beta_{1}},\, {\beta_{2}},\, \cdots,\, {\beta_{d}})$ in $\bC[x_{ij} : 1\,\le\, i \,\le\, s, 1 \,\le\, j \,\le\, d]$ of size $d^{s \,-\,1}$.
\end{proposition}

We would like to note that the statement in \cite[Lemma 1.5]{HERZOG1991187} concerns the existence and size of a matrix factorization of the generic form $g=\sum_{i=1}^s \prod_{j=1}^d x_{ij}$. This statement is used to prove the existence of a matrix factorization of $f \in I^{d}$, where $I$ is an ideal of a commutative ring $R$ \cite[Theorem 1.2]{HERZOG1991187}. More generally, the arguments in  \cite[Lemma 1.5]{HERZOG1991187} provide a proof of the Proposition \ref{existencematrix}. Observe that a homogeneous polynomial of the form $\sum_{i=1}^{s}\prod_{j=1}^{d}a_{ij}x_{ij}$ may contain fewer number of terms than the generic form $\sum_{i=1}^s \prod_{j=1}^d x_{ij}$, since some of the coefficients $a_{ij}$ may be zero. From the proof of the Proposition \ref{existencematrix}, we deduce that the size of a matrix factorization of $g = \sum_{i=1}^{s}\prod_{j=1}^{d}a_{ij}x_{ij}$ can  be strictly less than $d^{s-1}$. 

For completeness and clarity, we provide a proof of the Proposition \ref{existencematrix} (for more details, see \cite[Lemma 1.5]{HERZOG1991187}). There is a bijective correspondence between matrix factorizations of a homogeneous polynomial $f$ of degree $d$ and $\bZ/d\bZ$ graded modules over the universal Clifford algebra of $f$ \cite[Theorem 1.3]{BFb0082014}. The crucial element in the proof is that matrix factorization for the sum of two polynomials $g_{1}$ and $g_{2}$ corresponds to a certain tensor product between the generalized Clifford algebras associated to $g_{1}$ and $g_{2}$. 

\begin{proof}[Proof of Proposition \ref{existencematrix}]
 We will induct on $s$. For $s=1$, the monomial $\prod_{j=1}^{d}a_{1j}x_{1j}$ will have a matrix factorization $(a_{11}x_{11})(a_{12}x_{12}) \cdots (a_{1d}x_{1d})$ of size $1$. Suppose that $\tilde{g}=\sum_{i=1}^{s-1}\prod_{j=1}^{d}a_{ij}x_{ij}$ has a matrix factorization $\tilde{\beta} = (\tilde{\beta_{1}}, \tilde{\beta_{2}}, \cdots, \tilde{\beta_{d}})$. Let $\zeta$ be a primitive $d^{th}$ root of unity, and $p=d^{s-2}$. Let $\beta_l$ be the following matrix\\
 \begin{center}
 
 $\beta_l = \begin{bmatrix}
 	\tilde{\beta_{l-1}} & \zeta^{l-1}a_{s1}x_{s1} \id_{p} & 0 & 0 & \cdots & \cdots & 0 \\
 	0 & \tilde{\beta_{l-2}} & \zeta^{l-2}a_{s2}x_{s2}\id_{p} & 0 & \cdots & \cdots & 0 \\
 	0 & 0 & \tilde{\beta_{l-3}} & \cdots & \cdots &  & \vdots \\
 	\vdots & 0 &  &  & \vdots & \vdots & \vdots \\
 	\vdots & \vdots &  &  & \vdots & \vdots & 0 \\
 	0 & \vdots &  &  &  & \vdots & \zeta^{l-d+1}a_{sd-1}x_{sd-1}\id_{p} \\
 	\zeta^{l-d}a_{sd}x_{sd}\id_{p} & 0 & 0 & \cdots & \cdots & 0 & \tilde{\beta_{l-d}} 
 \end{bmatrix}  $\\

 \end{center}
Here $\tilde{\beta_{i}} = \tilde{\beta_{j}}$ for $j \equiv i (\text{mod}\; d)$, and $\id_p$ is the identity matrix of size $p$. Let $\tilde{\beta}$ corresponds to the graded module $\tilde{M}$ over the Clifford algebra $c(\tilde{g})$. Let $N$ be the trivial module over the Clifford algebra $c(a_{s1}x_{s1} \cdot a_{s2}x_{s2} \cdots a_{sd}x_{sd})$. Then it can be seen that for a suitably chosen basis of $\tilde{M}$ and $N$, the tensor product $\tilde{M} \,\hat{\otimes}\, N$ is the module associated with the matrix factorization $({\beta_{1}} \,\cdot\, {\beta_{2}} \,\cdot\, \,\cdots\,  {\beta_{d}})$. Thus ${\beta} \,=\, ({\beta_{1}},\, {\beta_{2}},\, \cdots,\, {\beta_{d}})$ is the matrix factorization of $\tilde{g}\,+\, a_{s1}x_{s1} \,\cdot\, a_{s2}x_{s2} \,\cdots\, a_{sd}x_{sd}\,=\, g$. 
\end{proof}

\begin{remark}
	\label{keymfremark}
	The crucial takeaway from the proof of the Proposition \ref{existencematrix} is as follows. Let $f$ be a homogeneous polynomial of degree $d$ and 
	$$B_1 \,\cdot\, B_2 \,\cdots\, \,\cdot\, B_d \,=\, f \,\cdot\, \id_{m \times m}$$
	be a matrix factorization of $f$ of size $m$. Then the arguments in the proof of the Proposition \ref{existencematrix} shows that the polynomial $(z_1 \,\cdot\, z_2 \,\cdots\, \,\cdot\, z_d \,+\, f)$ admits a matrix factorization of size $d \cdot m$
	$$Q_1 \,\cdot\, Q_2 \,\cdots\, \,\cdot\, Q_d \,=\, (z_1 \,\cdot\, z_2 \,\cdots\, \,\cdot\, z_d \,+\, f) \,\cdot\, \id_{dm \times dm}.$$
	Let $\{C_i\}_{i=1}^d$ and $\{D_i\}_{i=1}^d$ be two matrix factorizations of size $n_1$ and $n_2$ of two degree $d$ homogeneous polynomials $g$ and $h$ respectively. Then, more generally, similar arguments show that, there exists a matrix factorization $\{A_i\}_{i=1}^d$ of size $d \, \cdot\, n_1 \,\cdot\, n_2$ of the polynomial $(g \,+\, h)$ (see \cite[Theorem 4.2.5]{CostaMiróRoigPonsLlopis}).
\end{remark}

\begin{remark} \label{matrixfactorizationsize}
    Given a matrix factorization of size $m$ of a homogeneous polynomial $f$ of degree $d$
    $$B_1 \cdot B_2 \cdot \cdots \cdot B_d=f \cdot \id_{m \times m},$$
    one can always construct a matrix factorization of $f$ of size $n \cdot m$ for any $n \in \bN$ as follows. Consider the matrices
    $$A_i \,=\, \begin{bmatrix}
B_i & 0 & 0 & \cdots & \cdots & 0 & 0 \\
0 & B_i & 0 & \cdots & \cdots & 0 & 0 \\
\vdots & 0 & B_i & \cdots & \cdots & 0 & \cdots \\
\vdots & \vdots & 0 & \cdots & \cdots &  & \cdots \\
0 & \cdots & \cdots & 0 & \cdots & 0 & \cdots \\
0 & 0 & \cdots & \cdots & 0 & B_i & 0 \\
0 & 0 & \cdots & \cdots & \cdots & 0 & B_i 
\end{bmatrix}  $$
such that each block matrix $B_i$ on the diagonal is of size $m$, and the remaining entries are zero. Thus, each $A_i$ is a matrix of size $mn \times mn$. It follows that 
$$A_1 \cdot A_2 \cdot \cdots \cdot A_d=f \cdot \id_{mn \times mn}.$$
\end{remark}

\subsection{Cyclic coverings of projective spaces} \label{cycliccovering}
Let $X$ be a projective variety of dimension $n$. Let $\bL$ be a line bundle on $X$, and $s \in H^0(X, \bL^{\otimes d \cdot k})$. Let $\bM$ be the line bundle $\bL^{\otimes k}$. Then $s$ is in $H^0(X, {\bM^{\otimes d}})$. By a cyclic covering of $X$ of degree $d$ for the data $(\bM, s)$, we mean the following construction \cite[Section 4.1.B]{lazarsfeld2004positivity}\\

Let $\pi\,:\, Z \,=\, \text{spec}(\text{sym}(\bM^{-1})) \,\longrightarrow\, X$ be the total space of rank $1$ locally free sheaf $\bM$. Then the line bundle $\pi^* \bM$ over $Z$ has a tautological section $T$. The vanishing $\{T=0\}$ defines the zero section $z:X \longrightarrow Z$. 
We define the hypersurface $Y$ as
$$Y \,=\, \text{Zero}(T^d-\pi^*s) \,\subset\, Z$$
Then the restriction $\pi:\, Y \longrightarrow X$ is the associated cyclic covering of degree $d$. The section $T^d-\pi^*s$ defines a map $\cO_Z \longrightarrow \pi^*{\bM^{\otimes d}}$ and hence the injective map $\pi^*{\bM^{\otimes -d}} \longrightarrow \cO_Z$. The ideal sheaf of $Y$ is the image of this map and is isomorphic to $(\bM^{-d} \oplus \bM^{-d-1} \oplus \cdots)$ as an $\mathcal{O}_{X}$ module. Thus, $\pi_* \cO_Y$ is the following 

\begin{equation}
	\pi_*\mathcal{O}_Y \,\isom\, \mathcal{O}_{X} \oplus \bM^{-1} \oplus\bM^{-2} \oplus \cdots \oplus \bM^{-d+1}.
\end{equation}

When $X = \bP^n$, we can also construct the cyclic covering as follows. Let $B$ be a hypersurface defined by a polynomial $f$ of degree $d \cdot k$, for $k \in \bN$ i.e., $f \in H^0(\bP^n, \mathcal{O}_{\bP^n}(d \cdot k))$. We consider the weighted projective space $\bP(1, 1, \cdots, 1, k)$ of dimension $(n+1)$. Let $Y$ be the hypersurface defined by $(t^d-f)$ in $\bP(1, 1, \cdots, 1, k)$. For $k=1$, the space $\bP(1, 1, \cdots, 1, k)$ is $\bP^{n+1}$, and for $k>1$,  the space $\bP(1, 1, \cdots, 1, k)$ has only one singular point $(0,0, \cdots, 0, 1)$. Projection from this point will define a cyclic covering $\pi\,:\, Y \longrightarrow X$. 

\begin{remark}
	\begin{itemize}
		\label{singcyclic}
		\item [(1)] The complement of $(0,0, \cdots, 0, 1)$ in $\bP(1, 1, \cdots, 1, k)$ can be identified with the total space $Z=\text{spec}(\text{sym}(\mathcal{O}_{\bP^n}(-k)))$. The variable $t$ restricted to $Z$ is the tautological section $T$. Thus, the cyclic covering $\pi: Y \longrightarrow \bP^n$ in this construction coincides with the previous one.
		 
		\item [(2)]Locally, $Y$ is the $d^{th}$ root construction of a regular section $s$ of an affine open subscheme $\text{spec}(A)$ of $X$. The local structure of $Y$ is $\text{spec}(\frac{A[t]}{t^d-s})$ with the natural map $A \longrightarrow \frac{A[t]}{t^d-s}$. Thus, cyclic covering is a unique construction. 
		
		\item [(3)] Let $X$ be a smooth variety. Then, using the Jacobian criterion on the local description, we see that $B = zero(s)$ is a smooth hypersurface if and only if $Y$ is smooth. 
	\end{itemize}
\end{remark}

\section{Existence of Ulrich bundles on cyclic coverings of $\mathbb{P}^n$}
\label{three}
In this section, we will prove the existence of \ul bundles on cyclic coverings of $\bP^n$ of arbitrary degree. We will employ a sheaf theoretic argument using matrix factorization of polynomials.
\subsection{\ul Sheaves } Let $Y$ be a projective variety. Let $Y \longhookrightarrow \bP^{N}$ be an embedding given by a very ample line bundle $H$ on $Y$.
\begin{definition}
	\label{uldefin}
	Let $E$ be a coherent sheaf on $Y$ such that
	$\text{dim}(\text{supp}(E))\,=\, k$. Then $E$ is called an Ulrich sheaf if
	the following cohomology groups vanish: $H^{i}(Y, \, E(d))=0$ for all
	$1 \le i \le (k-1)$ and all $d \,\in \, \mathbb{Z} $,
	$H^{0}(Y,\, E(j))=0$ for all $j <0$, and $H^{k}(Y,\, E(j))\,=\,0$ for
	all $j \ge -k$.
\end{definition} 
\ul sheaves have several equivalent characterizations, as follows:

\begin{proposition}\cite[Proposition 2.1]{eisenbud2003resultants}\label{Eisenbuddefinition}
	Let $E$ be a coherent sheaf on the projective variety $Y$ such that $\text{dim} \,\, Y \,=\, k \,>\,0$ and $\text{supp}\,(E)\,=\, Y$. Then the following conditions are equivalent: 
	\begin{itemize}
		\item [(D1)] Let $E$ be an \ul sheaf.
		\item [(D2)] The cohomology $H^{i}(Y,\, E(-i))\, =\,0$ for $i \, >\, 0$, $H^{i}(Y,\, E(-i-1)) \,=\,0$ for $i \, <\, k$.
		\item [(D3)]For some (respectively all) finite linear projections $\pi \,:\, Y \longrightarrow \bP^k$, the sheaf $\pi_* E$ is the trivial sheaf $\cO_{\bP^k}^{\oplus t}$ for some $t$. 
	\end{itemize}
\end{proposition}

For a survey on \ul bundles we refer to \cite{beauville2018introduction}.\\

Let $X \,\subseteq\, \bP^{n+1}$ be a hypersurface defined by an irreducible polynomial $g \in \bC[x_0,\, x_1,\, \cdots, x_{n+1}]$ of degree $d$. Note that $X$ may not be smooth. Let $\alpha=(\alpha_1, \alpha_2, \cdots, \alpha_d)$ be a matrix factorization of $g$ of size $m$. The existence of such a matrix factorization is guaranteed by Proposition \ref{existencematrix}. We define the following short exact sequence
\begin{equation}
	\label{projdim}
	\begin{tikzcd}[sep=scriptsize]
		0 && {\mathcal{O}_{\mathbb{P}^{n+1}}(-1)}^{\oplus m} && {\mathcal{O}_{\mathbb{P}^{n+1}}}^{\oplus m} && G_i && 0
		\arrow[from=1-1, to=1-3]
		\arrow["{\alpha_i}", from=1-3, to=1-5]
		\arrow[from=1-5, to=1-7]
		\arrow[from=1-7, to=1-9]
	\end{tikzcd},
\end{equation}
where $G_i$ is the sheaf $\coker(\alpha_i)$.

The sheaves $G_i$ are supported on $X$. Since $X$ might have singularities, we can not confirm that $G_{i }$ has a finite projective dimension. Thus even though it can be computed that $\text{depth}_{\cO_{X, x}} (G_i)_x=n$, $G_i$ may not be a vector bundle. Let $\pi : X \longrightarrow \bP^n$ be a finite covering of degree $d$ defined by a linear projection. We will prove that the sheaves $G_i$ are \ul  on $X$.

\begin{proposition}
	\label{keythm}
	The sheaves $\pi_*G_i$ are trivial vector bundles of rank $m$.
\end{proposition}

\begin{proof}
We will present two different proofs. The first involves verifying that $G_i$ satisfies the cohomological properties in Definition \ref{uldefin},  from which the triviality of  $\pi_*(G_i)$ follows as a consequence of Proposition \ref{Eisenbuddefinition}. The second proof is a direct geometric argument establishing the triviality of $\pi_*(G_i)$.\\
 
We begin by showing that $G_i$ satisfies the cohomological properties described in Definition \ref{uldefin}. Taking the long exact sequence in cohomology corresponding to \eqref{projdim}, we consider the following part of the exact sequence
 \begin{equation}
 	\label{lesmainthm}
 	\longrightarrow  H^k(\bP^{n+1}, \,\cO_{\bP^{n+1}}^{\oplus m}(t)) \longrightarrow H^k(X, \, G_i(t)) \longrightarrow H^{k+1}(\bP^{n+1},\, \cO_{\bP^{n+1}}^{\oplus m}(t-1)) \longrightarrow 
 \end{equation}
 If $1 \,\le k \,\le \text{dim}\,(X)\,-\,1\,=\, n-1$, then we see that $$H^k(\bP^{n+1},\, \cO_{\bP^{n+1}}^{\oplus m}(t))\,=\,0\,=\,H^{k+1}(\bP^{n+1},\, \cO_{\bP^{n+1}}^{\oplus m}(t-1))$$ for all $t \,\in\, \bZ$. Thus we will have $H^k(X, \, G_i(t))\,=\,0$. Also similarly it follows that $H^0(X, \, G_i(t))\,=\,0$ for $t \,<\,0$.
Since $H^n(\bP^{n+1}, \,\cO_{\bP^{n+1}}^{\oplus m}(t))\,=\,0$, putting $k \,=\, n$ in \eqref{lesmainthm} we will get the following exact sequence
$$0 \longrightarrow H^n(X,\, G_i(t)) \longrightarrow H^{n+1}(\bP^{n+1},\, \cO_{\bP^{n+1}}^{\oplus m}(t-1)) \longrightarrow $$ 
 By Serre duality $H^{n+1}(\bP^{n+1},\, \cO_{\bP^{n+1}}^{\oplus m}(t-1)) \cong H^0((\bP^{n+1},\, \cO_{\bP^{n+1}}^{\oplus m}(-t+1-n-2)))$. Thus for $t \,\ge -n$ we see that  $H^{n+1}(\bP^{n+1},\, \cO_{\bP^{n+1}}^{\oplus m}(t-1))=0$, and hence $H^n(X,\, G_i(t))\,=\,0$. Thus $G_i$ is an \ul sheaf. By Proposition \ref{Eisenbuddefinition} (D3) we will have $\pi_*(G_i)$ is trivial. Since \text{dim} $H^0(X, \, G_i)\,=\, m$, in particular $\pi_* (G_i) \,\cong\, \cO_{\bP^n}^{\oplus m}$.\\
 
 We now present a second proof, which directly establishes the triviality of $\pi_*(G_i)$ using geometric considerations. let $Q \in \bP^{n+1}$ be a point outside $X$, and $\bP^{n+1} \dashrightarrow \bP^n$ be the projection from $Q$. Let $\pi\,:\,X \,\longrightarrow\, \bP^n$ be the induced finite map. We consider the following diagram
 
\[\begin{tikzcd}
	& {\text{Bl}_{Q}\mathbb{P}^{n+1}} \\
	{} & {\mathbb{P}^{n+1}} \\
	X &&& {\mathbb{P}^n}
	\arrow["\rho", from=1-2, to=2-2]
	\arrow["p", from=1-2, to=3-4]
	\arrow[dashed, from=2-2, to=3-4]
	\arrow["i", hook, from=3-1, to=1-2]
	\arrow["i", hook', from=3-1, to=2-2]
	\arrow["\pi"', from=3-1, to=3-4]
\end{tikzcd}\]
  The inclusion $i \,:\, X \longhookrightarrow \bP^{n+1}$ lifts to $i \,:\, X \longhookrightarrow  {\text{Bl}_{Q}\mathbb{P}^{n+1}}$, and we have $\pi\,=\, p \circ i$. We apply the Fourier-Mukai transform $ p_* \,\circ\, \rho^* $ on \eqref{projdim}
 \begin{equation}
 	0 \longrightarrow p_* \circ \rho^*{\mathcal{O}_{\mathbb{P}^{n+1}}(-1)}^{\oplus m} \longrightarrow p_* \circ \rho^* {\mathcal{O}_{\mathbb{P}^{n+1}}}^{\oplus m} \longrightarrow \pi_* G_i \longrightarrow R^1p_* \circ \rho^*{\mathcal{O}_{\mathbb{P}^{n+1}}(-1)}^{\oplus m}  \longrightarrow 
 \end{equation}
 We can see that the Fourier-Mukai transform $p_* \circ \rho^*$ restricts to $p_* \circ i_* =\pi_*$ on $X$. The map $p:{\text{Bl}_{Q}\mathbb{P}^{n+1}} \longrightarrow {\mathbb{P}^n}$ is a $\mathbb{P}^1$ bundle. The fibers can be identified with the strict transform of the lines passing through the point $Q$
 
 \[\begin{tikzcd}
 	{p^{-1}\{x\}=\text{Bl}_{Q}\mathbb{P}^{1}} & {\text{Bl}_{Q}\mathbb{P}^{n+1}} \\
 	{\mathbb{P}^1} & {\mathbb{P}^{n+1}}
 	\arrow[hook, from=1-1, to=1-2]
 	\arrow["{\text{id}}"', from=1-1, to=2-1]
 	\arrow["\rho", from=1-2, to=2-2]
 	\arrow[hook, from=2-1, to=2-2]
 \end{tikzcd}\]
 Since $ {\rho^*{\mathcal{O}_{\mathbb{P}^{n+1}}(-1)}}_{|{p^{-1}\{x\}}} \cong {\mathcal{O}_{\mathbb{P}^{n+1}}(-1)}_{|\bP^1}=\cO_{\bP^1}(-1)$, we see that $ {\rho^*{\mathcal{O}_{\mathbb{P}^{n+1}}(-1)}}$ is the relative $\cO(-1)$ for $p$. Thus, the Fourier-Mukai transform $p_* \circ \rho^*{\mathcal{O}_{\mathbb{P}^{n+1}}(-1)}^{\oplus m} =0$, and $R^1p_* \circ \rho^*{\mathcal{O}_{\mathbb{P}^{n+1}}(-1)}^{\oplus m} =0$. Also, we see that $p_* \circ \rho^* ({\mathcal{O}_{\mathbb{P}^{n+1}}}^{\oplus m}) \,\cong\, \cO_{\bP^n}^{\oplus m}$. Hence, we get $\pi_* G_i \,\cong\, \cO_{\bP^n}^{\oplus m}$.
\end{proof}

As a consequence, we will get the main result of this section

\begin{theorem}
	\label{veroneseargumenttheorem}
	There exist \ul vector bundles on every smooth cyclic covering of  $\bP^n$. 
\end{theorem}

\begin{proof}
	We will generalize the method of \cite{kumar2023ulrichbundlesdoublecovers}. Let $\pi : Z \longrightarrow \mathbb{P}^{n}$ be a smooth cyclic covering of degree $d$. Then the branch locus $B$ is a smooth hypersurface in $\mathbb{P}^{n}$ defined by a homogeneous polynomial $g$ of degree $d \cdot k$ for some $k \in \bN$. We consider the Veronese embedding
	\begin{equation}
		\begin{tikzcd}
			{\mathbb{P}^n} && {\mathbb{P}^N}
			\arrow["{|\mathcal{O}_{\mathbb{P}^n}(k)|}", hook, from=1-1, to=1-3]
		\end{tikzcd}
	\end{equation} 
	Changing the expression of $g$ by using Veronese coordinates, we can construct a new polynomial $g'$ of degree $d$ with variables in Veronese coordinates with the property: Let $B' \subseteq \bP^N$ be the hypersurface defined by the degree $d$ homogeneous polynomial $g'$, then we will have $B' \bigcap \mathbb{P}^{n} \,=\, B$. We define the hypersurface $\tilde{Z} \,=\, \text{zero}(t^{d}-g') \,\subseteq\, \mathbb{P}^{N+1}$. Then projection from $(0, 0, \cdots, 1) \in \bP^{N+1}$ will define a degree $d$ cyclic covering $\tilde{\pi} \,:\, \tilde{Z}\, \longrightarrow\, \mathbb{P}^{N}$ whose branch divisor is $B'$. We will have the following Cartesian diagram
	\begin{equation}
		\begin{tikzcd}
			{\text{Z}} && {\tilde{Z}} \\
			{\mathbb{P}^n} && {\mathbb{P}^N}
			\arrow["{|\mathcal{O}_{\mathbb{P}^n}(k)|}", hook, from=2-1, to=2-3]
			\arrow["\pi", from=1-1, to=2-1]
			\arrow["{\tilde{\pi}}"', from=1-3, to=2-3]
			\arrow[hook, from=1-1, to=1-3]
		\end{tikzcd}
	\end{equation}
By Proposition \ref{existencematrix}, there exists a matrix factorization of $t^d-g'$. In particular, if $g'=\sum_{i=1}^{s}\prod_{j=1}^{d}a_{ij}x_{ij}$, then $t^d\,-\,g'$ will have a matrix factorization of size $d^{s}$ by Remark \ref{keymfremark}. By Proposition \ref{keythm}, there exist sheaves $G_i$ of rank $d^{s-1}$ on $\tilde{Z}$ such that $\tilde{\pi}_*G_i \cong \cO_{\bP^{N}}^{d^{s}}$. By the base change theorem \cite[\href{https://stacks.math.columbia.edu/tag/02KG}{Tag 02KG}]{stacks-project}, we get
$$\pi_*({G_i}_{|Z}) \cong (\tilde{\pi}_*G_i)_{|\bP^n} \cong \cO_{\bP^{n}}^{d^s}$$
\end{proof}

\begin{remark}
	\begin{itemize}
		\item [(1)] Let $\text{sing}(B')$ be the singular locus of $B'$. Using $(2)$ of Remark \ref{singcyclic}, we see $\tilde{Z}' = \tilde{Z} \setminus \tilde{\pi}^{-1}(\text{sing}(B'))$ is a smooth open subvariety. Since the branch locus $B \subset B' \setminus \text{sing}(B')$, we will have $Z \subset \tilde{Z}'$. For $z \in \tilde{Z}$, $\text{depth}_{\mathcal{O}_{Z, z}}{(G_{i})_{z}} = \text{dim}(\mathcal{O}_{{Z},z})$. If $z \in \tilde{Z}'$, then the $\text{p.d.}(G_{i})_{z}$ is finite, and hence it is zero. This implies $(G_{i})_{z}$ is a projective module for $z \in \tilde{Z}'$. Hence the restriction $(G_{i})_{|Z}$ is a vector bundle.
		
		\item [(2)] We can work with a cyclic covering $X$ which is not smooth but has irreducible branch divisor. In this case, we will get an \ul sheaf instead of an \ul bundle. 
	\end{itemize}
	\end{remark}

In the following example, we will apply the results of this section to compute the rank of an \ul bundle on a degree $3$ smooth, cyclic covering of $\bP^2$.

\begin{example}
	Let $\pi\,:\,X \longrightarrow \bP^2$ be a smooth, cyclic covering of $\bP^2$ of degree $3$. The branch locus $B$ of the covering $\pi$ is a smooth curve in $\bP^2$, given by the zero locus of a polynomial $F$ of degree $3k$. Using \cite[Theorem 5.1]{carlini2008complete} we can write
	$$F \,=\, B_{1}E_{1} \,+\, B_{2}E_{2}$$  
	where $B_{1}, B_{2}$ are homogeneous polynomials of degree $k$. Repeated applications of \cite[Theorem 5.1]{carlini2008complete} for $E_{1}, E_{2}$ will give the following
	$$F = B_1(G_1H_1+G_2H_2)+B_2(Q_1P_1+Q_2P_2)$$ where the $G_{i}$, $H_{i}$, and $P_i, Q_i$ are homogeneous polynomials of degree $k$. Using the Veronese embedding $|\cO_{\bP^2}(k)|\,:\,\bP^2 \,\longrightarrow\, \bP^N$, the polynomial $F$ can be expressed as a cubic polynomial $F'$ in the Veronese coordinates of $\bP^N$.  
	In this special case both $B_1(G_1H_1\,+\,G_2H_2)$ and $B_2(Q_1P_1\,+\,Q_2P_2)$ will have a matrix factorization of size $2 \,\cdot\, 1 \,\cdot 1\,=\,2$ by Remark \ref{keymfremark}. Thus, we will have a matrix factorization of size $3 \cdot 2 \cdot 2=12$ for $F'$ and of size $3\cdot 1 \cdot 12=36$ for $t^3\,+\,F'$ by Remark \ref{keymfremark}. Hence, the result of  \S \ref{three}   alone will show the existence of an \ul bundle of rank $\frac{36}{3}\,=\,12$.
	
	 In \S \ref{five}, we will see that the matrix factorization techniques alone will not give the best possible estimation of the rank of an \ul bundle on $X$. Rather, a combination of the matrix factorization techniques and Theorem \ref{extensionthm} will help us to minimize the rank of the \ul bundle on $X$.  In \S \ref{five}, we will prove that the rank can be reduced to $6$ in this case; see Corollary \ref{d=3}.
\end{example}

\section{On the construction of a relatively Ulrich bundle on an ambient variety from a relatively Ulrich bundle on a complete intersection subvariety}
\label{four}
 
In this section, we will define the notion of a relatively \ul bundle with respect to a covering map between varieties. Then, under a certain framework, we will prove that the existence of a relatively \ul bundle on a subvariety will give rise to a relatively \ul bundle on the ambient variety. This will provide an alternative approach to prove the existence of an \ul bundle for cyclic coverings of $\bP^n$ under the assumption that a relatively \ul bundle exists for cyclic coverings of complete intersection curves. The aim of the present and the next section is to obtain a  better estimation of the rank. We will start with the following definition. 

\begin{definition}\label{relulbun}
	Let $\pi:Z \longrightarrow W$ be a finite covering of degree $d$ between smooth projective varieties. A vector bundle $F$ on $Z$ of rank $r$ is called a \textbf{relatively \ul bundle} for the covering $\pi$ if $\pi_*F$ is isomorphic with $\cO_W^{\oplus d \cdot r}$. 
\end{definition}
Let $\pi:Z \longrightarrow W$ be an arbitrary finite covering map between smooth projective varieties. Let $\cO_W(1)$ be an ample line bundle on $W$. For a coherent sheaf $K_1$ on $Z$, we have the natural short exact sequence
\begin{equation}
\begin{tikzcd}
	0 & M_1 & \pi^*\pi_*K_1 & K_1 & 0
	\arrow[from=1-1, to=1-2]
	\arrow[from=1-2, to=1-3]
	\arrow["ev", from=1-3, to=1-4]
	\arrow[from=1-4, to=1-5]
\end{tikzcd}
\end{equation}
where $M_1$ be the kernel of the evaluation map $ev:\pi^*\pi_*K_1 \longrightarrow K_1$. Tensoring with $\pi^*\cO_{W}(l)$ we get the following
\begin{equation}
	\begin{tikzcd}
		\label{evaluationSES}
	0 & M_1 \otimes \pi^*\cO_{W}(l) & \pi^*\pi_*K_1 \otimes \pi^*\cO_{W}(l) & K_1 \otimes \pi^*\cO_{W}(l) & 0
	\arrow[from=1-1, to=1-2]
	\arrow[from=1-2, to=1-3]
	\arrow["ev \otimes \id", from=1-3, to=1-4]
	\arrow[from=1-4, to=1-5]
\end{tikzcd}
\end{equation}

\begin{remark}
	\label{evaluationlemma}
	It can be checked that the map $ev \otimes \id$ in \eqref{evaluationSES} is the evaluation map $ev:\pi^*\pi_*(K_1 \otimes \pi^*\cO_{W}(l)) \longrightarrow K_1 \otimes \pi^*\cO_{W}(l)$ in the sense that the following diagram commutes
	\[\begin{tikzcd}
		{\pi^*\pi_*(K_1 \otimes \pi^*\cO_{W}(l))} && {K_1 \otimes \pi^*\cO_{W}(l)} \\
		{\pi^*\pi_*K_1 \otimes \pi^*\cO_{W}(l)} && {K_1 \otimes \pi^*\cO_{W}(l)}
		\arrow["ev", from=1-1, to=1-3]
		\arrow["{\cong(\text{projection formula})}", from=1-1, to=2-1]
		\arrow["id", from=1-3, to=2-3]
		\arrow["ev \otimes \id", from=2-1, to=2-3]
	\end{tikzcd}\]
\end{remark}

The following is the main theorem of this section. 
 \begin{theorem}
 	\label{generalextensionthm}
 	Let $\pi\,:\, X \,\longrightarrow\, Y$ be a cyclic covering of degree $d$ of smooth projective varieties. Then we will have $\pi_* \cO_X \,\cong\, \cO_Y \,\oplus\, L^{-1} \,\oplus\, \cdots \,\oplus\, L^{-d+1}$ for some line bundle $L \,\in\, \text{Pic}(Y)$ such that the branch divisor $B \in |L^{\otimes d}|$. Let $W \,\in\, |L|$ be a smooth divisor such that $B \bigcap W$ is smooth(or equivalently, the intersection is transversal). We consider the following Cartesian diagram
 	\begin{center}
 		\begin{tikzcd}
 			{{Z}} && {{X}} \\
 			{W} && {Y}
 			\arrow[hook, from=2-1, to=2-3]
 			\arrow["\pi'", from=1-1, to=2-1]
 			\arrow["\pi"', from=1-3, to=2-3]
 			\arrow[hook, from=1-1, to=1-3]
 		\end{tikzcd}
 	\end{center}
 Then $\pi'\,:\, Z \,\longrightarrow\, W$ is a cyclic covering of degree $d$ of smooth projective varieties. Further we will assume $H^{1}(Y,\, L^{\otimes j})\,=\,0$ for all $j \,\in\, \bZ$. 
 	  If $F$ is a  rank $r$ \emph{relatively \ul bundle} on $Z$ with respect to $\pi'$, then there exists a \emph{relatively \ul bundle} $E$ of rank $d \cdot r$ on $X$ with respect to $\pi$.
 \end{theorem}
\begin{proof}
	\textbf{\underline{Step 1}}: \textbf{First modification} We consider the evaluation map
	$$\pi'^*\pi'_* F \longrightarrow F.$$
	Since, $\pi'_*F \,\cong\, \cO_W^{\oplus d \cdot r}$, we get the surjective morphism $\cO_Z^{\oplus d \cdot r} \longrightarrow F$. Composing with the structure map $\cO_X \longrightarrow \cO_Z$ we get the following
	$$\phi\,:\,\cO_X^{\oplus d \cdot r} \,\longrightarrow\, F.$$
	Let $K_1'$ be the kernel of $\phi$:
	\begin{equation}
		\begin{tikzcd}
			0 & {K_1'} && {\cO_X^{\oplus d \cdot r}} && F & 0
			\arrow[from=1-1, to=1-2]
			\arrow[from=1-2, to=1-4]
			\arrow["\phi", from=1-4, to=1-6]
			\arrow[from=1-6, to=1-7].
		\end{tikzcd}
	\end{equation} 
	Since $\pi$ is a finite map, for any coherent sheaf $G$ on $X$, $R^i \pi_* G =0$ for all $i >0$. Thus applying $\pi_*$, and using the identification $\pi_*\cO_X \cong \cO_Y \oplus L^{-1} \oplus \cdots \oplus L^{-d+1}$, we get the following short exact sequence 
	\begin{equation}
		\label{firstmod}
		\begin{tikzcd}
			0 & {\pi_*K_1'} & {\bigg(\cO_Y \oplus L^{-1} \oplus \cdots \oplus L^{-d+1}\bigg)^{\oplus d \cdot r} } & {\cO_W^{\oplus d \cdot r}} & 0
			\arrow[from=1-1, to=1-2]
			\arrow[from=1-2, to=1-3]
			\arrow["{\pi_*\phi}", from=1-3, to=1-4]
			\arrow[from=1-4, to=1-5]
		\end{tikzcd}
	\end{equation}
	
	\textbf{\underline{Claim}:}	The subbundle ${\bigg(L^{-1} \oplus \cdots \oplus L^{-d+1}\bigg)^{\oplus d \cdot r} }$ under the map $\pi_*\phi$ will map to zero.\\
	\textbf{proof of the claim:} We consider the following commutative diagram
	\[\begin{tikzcd}
		{\pi^*\pi_*(\cO_X^{\oplus d \cdot r})} & {\cO_X^{\oplus d \cdot r}} \\
		& F
		\arrow["ev", from=1-1, to=1-2]
		\arrow["{\phi \circ ev}"', from=1-1, to=2-2]
		\arrow["\phi", from=1-2, to=2-2]
	\end{tikzcd}\]
	The map $\phi \circ ev \in Hom({\pi^*\pi_*(\cO_X^{\oplus d \cdot r})}, F)$ induces the map $\pi_* \phi  \in Hom(\pi_*\cO_X^{\oplus d \cdot r}, {\cO_W^{\oplus d \cdot r}})$ via the adjoint isomorphism
	$$ Hom({\pi^*\pi_*(\cO_X^{\oplus d \cdot r})},\, F) \,\cong\, Hom(\pi_*\cO_X^{\oplus d \cdot r},\, \pi_*F) \,\cong\, Hom(\pi_*\cO_X^{\oplus d \cdot r},\, {\cO_W^{\oplus d \cdot r}}).$$
	Let $M$ be the kernel of the map $ev\,:\, {\pi^*\pi_*(\cO_X^{\oplus d \cdot r})} \,\longrightarrow\, {\cO_X^{\oplus d \cdot r}}$:
	\begin{equation} \label{neweq}
		0 \,\longrightarrow\, M \,\longrightarrow\, {\pi^*\pi_*(\cO_X^{\oplus d \cdot r})} \,\longrightarrow\, {\cO_X^{\oplus d \cdot r}} \longrightarrow 0.
	\end{equation} 
	Since $H^0(ev)$ induces a surjective map on the zeroth cohomology, and both sides have $d \cdot r$ sections, $H^0(ev)$ is an isomorphism. We have $$H^1(X,\, {\pi^*\pi_*(\cO_X^{\oplus d \cdot r})})\,=\,H^1(Y,\, \pi_*{\pi^*\pi_*(\cO_X^{\oplus d \cdot r})})\,=\,H^1(Y,\, \pi_*(\cO_X^{\oplus d \cdot r}) \otimes \pi_*\cO_X)\,=\,0,$$ since $H^1(Y,\, L^{\otimes j})\,=\,0$ for all $j \,\in\, \bZ$, by assumption. Thus, taking the long exact sequence in cohomology for the short exact sequence \eqref{neweq}, we get $H^1(X, \, M)\,=\,0$. Thus, we get $Ext^1({\cO_X^{\oplus d \cdot r}},\, M)\,=\,H^1(X,\, M)^{\oplus d \cdot r} \,=\,0$. Hence $M\,=\,\pi^*{\bigg( L^{-1} \oplus \cdots \oplus L^{-d+1}\bigg)^{\oplus d \cdot r} }$. This proves the claim.\\ 
	
	Since the ideal sheaf of $W$ is $L^{-1}$, from \eqref{firstmod} it follows that $$\pi_{*}K_1' \,\cong\, (L^{-1})^{\oplus 2 d \cdot r} \,\oplus\, {\bigg(L^{-2} \,\oplus\, \cdots \,\oplus\, L^{-d+1}\bigg)^{\oplus d \cdot r}}.$$ 
	Using the projection formula, we get $$\pi_{*}(K_1' \otimes \pi^{*} L) \,\cong\, \cO_Y^{\oplus 2 d \cdot r} \,\oplus\, {\bigg( L^{-1} \oplus \cdots \oplus L^{-d+2}\bigg)^{\oplus d \cdot r}}.$$

	\textbf{\underline{Step2}}:\; \textbf{Second modification} Let $K_1$ be the vector bundle $K_1' \otimes \pi^* L$. We will show that there exists a surjective morphism $K_1 \,\longrightarrow\, F^{\oplus 2}$ with kernel $K_2'$ such that pushforward is given by  $$\pi_*K_2'\,\cong\,(L^{-1})^{\oplus 3 d \cdot r} \,\oplus\, {\bigg( L^{-2} \oplus \cdots \oplus L^{-d+2}\bigg)^{\oplus d \cdot r}}.$$
	
	Let $M_1$ be the kernel of the evaluation map $ev\,:\,\pi^*\pi_*K_1 \,\longrightarrow\, K_1$ which fits into the following short exact sequence
	\begin{equation}
		\label{secondmod}
		\begin{tikzcd}
			0 & M_1 & \pi^*\pi_*K_1 & K_1 & 0
			\arrow[from=1-1, to=1-2]
			\arrow[from=1-2, to=1-3]
			\arrow["ev", from=1-3, to=1-4]
			\arrow[from=1-4, to=1-5].
		\end{tikzcd}
	\end{equation}
	 We will show that $H^1(X, \, M_1 \otimes \pi^*L^{\otimes j})=0$ for all $j \,\in\, \bZ$ as follows. Tensoring with $\pi^*L^{\otimes j}$ we get the following
	 \begin{equation}  \label{newlyaddedeqn}
	\begin{tikzcd}
		0 & M_1 \otimes \pi^*L^{\otimes j} & \pi^*\pi_*K_1 \otimes \pi^*L^{\otimes j} & K_1 \otimes \pi^*L^{\otimes j} & 0
		\arrow[from=1-1, to=1-2]
		\arrow[from=1-2, to=1-3]
		\arrow["ev \otimes \id", from=1-3, to=1-4]
		\arrow[from=1-4, to=1-5].
	\end{tikzcd}
	\end{equation}
By Remark \ref{evaluationlemma}, $ev \otimes \id$ is the evaluation map $ev\,:\, \pi^*\pi_*(K_1 \otimes \pi^*L^{\otimes j}) \,\longrightarrow\,  (K_1 \otimes \pi^*L^{\otimes j})$.
	Thus the induced map $H^0(ev)\,:\,H^0(X, \, \pi^*\pi_*(K_1 \otimes \pi^*L^{\otimes j})) \,\longrightarrow\,  H^0(X, \, K_1 \otimes \pi^*L^{\otimes j})$ is surjective. Also from the assumption in Theorem \ref{generalextensionthm}, it follows $H^1(X, \, \pi^*\pi_*(K_1 \otimes \pi^*L^{\otimes j}))\,=\,0$. Taking the long exact sequence in cohomology for \eqref{newlyaddedeqn}, we get $H^1(X, \, M_1 \otimes \pi^*L^{\otimes j})\,=\,0$ for all $j \,\in\, \bZ$.\\
	
	Applying $\pi_*$ to the sequence \eqref{secondmod} we get
	\begin{equation}
		\label{thirdmod}
		\begin{tikzcd}
			0 & \pi_*M_1 & \pi_*\pi^*\pi_*K_1 & \pi_*K_1 & 0
			\arrow[from=1-1, to=1-2]
			\arrow[from=1-2, to=1-3]
			\arrow[from=1-3, to=1-4]
			\arrow[from=1-4, to=1-5]
		\end{tikzcd}
	\end{equation}
	From the above consideration, we get
	\footnotesize
	\begin{eqnarray}
		Ext^1(\pi_*K_1,\, \pi_*M_1)&=& Ext^1(\cO_Y^{\oplus 2 d \cdot r} \oplus {\bigg( L^{-1} \oplus \cdots \oplus L^{-d+2}\bigg)^{\oplus d \cdot r}},\, \pi_*M_1)\\
		&=& H^1(Y,\, \pi_*M_1)^{\oplus 2 d \cdot r} \oplus \bigg(\oplus_{j=1}^{(d-2)} H^1(Y,\, \pi_*(M_1 \otimes \pi^*L^{\otimes j}))\bigg)^{\oplus d \cdot r}\\
		&=&0
	\end{eqnarray}
\normalsize
	Thus, the sequence \eqref{thirdmod} splits. We fix a splitting isomorphism $\pi_*\pi^*\pi_*K_1 \,\cong\, \pi_*M_1 \,\oplus\, \pi_*K_1$. Let $f$ be the map $f\,:\,\pi_*K_1 \,\longrightarrow\,  \cO_W^{\oplus 2 d \cdot r}$ which sends $\cO_{Y}^{\oplus 2 d \cdot r}$ to $\cO_W^{\oplus 2 d \cdot r}$, and sends the rest of the subbundles to zero. Let $\tilde{f}\,:\,\pi_*\pi^*\pi_*K_1 \,\longrightarrow\, \cO_W^{\oplus 2 d \cdot r}$ be the surjective map such that $\tilde{f}_{|\pi_*M_1}\,=\,0$, and $\tilde{f}_{|\pi_*K_1}\,=\,f$. Due to the adjoint isomorphism $Hom(\pi_*K_1,\, \cO_W^{\oplus 2 d \cdot r}) \,\cong\, Hom(\pi^*\pi_*K_1,\, F^{\oplus 2})$, the map $f$ will give rise to a map $g\,:\,\pi^*\pi_*K_1 \,\longrightarrow\,  F^{\oplus 2}$. Composing with sequence \eqref{secondmod}, we get the following diagram
	\[\begin{tikzcd}
		M_1 & \pi^*\pi_*K_1 \\
		& F^{\oplus 2}
		\arrow[from=1-1, to=1-2]
		\arrow["h", from=1-1, to=2-2]
		\arrow["g", from=1-2, to=2-2]
	\end{tikzcd}\]
The map $h\,=\,0$ can be seen as follows.	Applying $\pi_*$ to this diagram, we get  
	\[\begin{tikzcd}
		\pi_*M_1 & \pi_*\pi^*\pi_*K_1 \\
		& \pi_*F^{\oplus 2}=\cO_W^{\oplus 2 d \cdot r}
		\arrow[from=1-1, to=1-2]
		\arrow["\pi_*h", from=1-1, to=2-2]
		\arrow["\pi_*g", from=1-2, to=2-2]
	\end{tikzcd}\]
	Since the evaluation map $\pi^*\pi_*\pi^*\pi_*K_1 \,\longrightarrow\, \pi^*\pi_*K_1$ splits, from the functoriality of the adjoint isomorphism $Hom(\pi^*\pi_*\pi^*\pi_*K_1,\, F^{\oplus 2}) \,\cong\, Hom(\pi_*\pi^*\pi_*K_1,\, \cO_W^{\oplus 2 d \cdot r})$, it follows that $\pi_*g\,=\,\tilde{f}$. This implies that $\pi_*h\,=\,0$. Since $\pi$ is a finite map, $h\,=\,0$. From the sequence \eqref{secondmod}, $g$ will induce a map $\psi\,:\,K_1 \,\longrightarrow\, F^{\oplus 2}$. Again from the functoriality of the construction, it follows that $\pi_*\psi$ is the map $f\,:\,\pi_*K_1 \,\longrightarrow\,  \cO_W^{\oplus 2 d \cdot r}$ which sends $\cO_{Y}^{\oplus 2 d \cdot r}$ to $\cO_W^{\oplus 2 d \cdot r}$, and the remaining summands ${\bigg( L^{-1} \oplus \cdots \oplus L^{-d+2}\bigg)^{\oplus d \cdot r}}$ to zero. Let $K_2'$ be the kernel of $\psi$. Then it follows that $\pi_*K_2'\,\cong\,(L^{-1})^{\oplus 3 d \cdot r} \oplus {\bigg( L^{-2} \oplus \cdots \oplus L^{(-d+2)}\bigg)^{\oplus d \cdot r}}$.
	
	Let $K_2\,=\,K_2' \otimes \pi^*L$, and thus $\pi_*K_2\,\cong\,\cO_{Y}^{\oplus 3 d \cdot r} \oplus {\bigg( L^{-1} \oplus \cdots \oplus L^{(-d+3)}\bigg)}^{\oplus d \cdot r}$. 
	
	\textbf{\underline{Step 3}}:\; \textbf{Inductive step} We assume that there exists a vector bundle $K_i$ on $X$ such that 
	\begin{equation} \label{ind1}
		\pi_* K_i \,\cong\, \cO_Y^{\oplus (i+1)d \cdot r} \oplus \big(L^{-1} \oplus L^{-2} \oplus \cdots \oplus L^{(-d+i+1)}\big)^{\oplus d \cdot r}.
	\end{equation}
	
	Then we can construct a vector bundle $K_{i+1}$ on $X$ satisfying the following 
	$$\pi_* K_{i+1} \,\cong\, \cO_Y^{\oplus (i+2)d \cdot r} \oplus \big(L^{-1} \oplus L^{-2} \oplus \cdots \oplus L^{(-d+i+2)}\big)^{\oplus d \cdot r}.$$
	
	\begin{proof}[Proof of the Inductive step]  \let\qed\relax
		We will present the main steps in the construction of $K_{i+1}$. As the construction closely parallels that of Step 2, we refer the reader to Step 2 for certain technical details. We divide the proof into the following four steps.
		\begin{enumerate}
			\item [(1)] Let $M_i$ be kernel of the evaluation map $\text{ev}: \pi^* \pi_* K_i \longrightarrow K_i$:
			\begin{equation} \label{ind2}
				0 \longrightarrow M_i \longrightarrow \pi^* \pi_* K_i  \longrightarrow K_i \longrightarrow 0
			\end{equation}  
			Using arguments similar to the first paragraph of Step 2, it follows that 
			 \begin{equation} \label{ind3}
				H^1(X, \, M_i \otimes \pi^* L^{\otimes j})=0 \, \, \text{for all} \,\,  j \in \bZ.
			\end{equation} 
Applying  $\pi_*$ to \eqref{ind2}, we get the following short exact sequence			
			
			\begin{equation} \label{ind4}
		0 \longrightarrow \pi_* M_i \longrightarrow \pi_* \pi^* \pi_* K_i  \longrightarrow \pi_* K_i \longrightarrow 0
			\end{equation}
			Using \eqref{ind3}, we will get $Ext^1(\pi_* K_i, \, \pi_* M_i)=0$. Hence, the sequence \eqref{ind4} splits.

			\item [(2)]  Using the decomposition \eqref{ind1}, we define the map $f: \pi_*K_i \longrightarrow  \cO_W^{\oplus (i+1) d \cdot r}$ satisfying the  following two conditions 
			\begin{enumerate}
				\item [(i)] the restriction $f \vert \cO_Y^{\oplus (i+1) d \cdot r}$ is the natural surjective map $\cO_Y^{\oplus (i+1) d \cdot r}  \longrightarrow \cO_W^{\oplus (i+1) d \cdot r}$ induced by the closed immersion $W \longhookrightarrow Y$. 
				
				\item [(ii)] the restriction $f \vert  \big(L^{-1} \oplus L^{-2} \oplus \cdots \oplus L^{(-d+i+1)}\big)^{\oplus d \cdot r}\,=\, 0$.
			\end{enumerate}
			Using the adjoint isomorphism 
			$$\Hom (\pi^*\pi_* K_i, \, F^{\oplus (i+1)}) \cong \Hom (\pi_*K_i, \, \pi_* F^{\oplus (i+1)}) \cong \Hom (\pi_*K_i, \, \cO_W^{\oplus (i+1)d \cdot r}),$$
			we choose the morphism $g:\, \pi^*\pi_* K_i \longrightarrow \, F^{\oplus (i+1)}$ which maps to $f$. In the next step, we will prove that $M_i \subseteq \ker (g)$. 
			
			\item [(3)] We consider the following diagram
			\begin{equation} \label{ind5}
				\begin{tikzcd}
				0 & {\pi^*\pi_*M_i} && {\pi^*\pi_*\pi^*\pi_* K_i } && {\pi^*\pi_* K_i } & 0 \\
				&&&&& {F^{\oplus (i+1)}}
				\arrow[from=1-1, to=1-2]
				\arrow[from=1-2, to=1-4]
				\arrow["{\text{ev}}", from=1-4, to=1-6]
				\arrow["{g \circ\text{ev}}"', from=1-4, to=2-6]
				\arrow[from=1-6, to=1-7]
				\arrow["g", from=1-6, to=2-6]
			\end{tikzcd}
			\end{equation}
			
			From functoriality, the map $\pi_*(g)$ is given by the image of $g \circ \text{ev}$ under the adjoint isomorphism
			\begin{equation} \label{ind6}
				\Hom ({\pi^*\pi_*\pi^*\pi_* K_i } , \, {F^{\oplus (i+1)}}) \cong \Hom (\pi_*\pi^*\pi_* K_i , \, \cO_W^{\oplus (i+1)d \cdot r}).
			\end{equation}
			
			 From the splitting of the sequence \eqref{ind4}, it follows that the horizontal sequence in \eqref{ind5} splits. We fix a splitting isomorphism
			$$ {\pi^*\pi_*\pi^*\pi_* K_i } \cong {\pi^*\pi_*M_i} \oplus {\pi^*\pi_* K_i }.$$
			Under this splitting, the map $g \circ \text{ev}$ takes  ${\pi^*\pi_*M_i}$ to zero. Hence, from this splitting and the definition of $g$ in $(2)$,  it follows that $\pi_* (g): \pi_*\pi^*\pi_* K_i \longrightarrow \, \cO_W^{\oplus (i+1)d \cdot r}$ is precisely the map satisfying $\pi_* (g) \vert \pi_* K_i \, =\, f$ and $\pi_*(g) \vert \pi_* M_i=0$. Since $\pi$ is a finite map, it follows that $M_i \subseteq \ker (g)$. For more details, we refer to Step 2.
			
			\item [(4)] Hence, the map $g:\, \pi^*\pi_* K_i \longrightarrow \, F^{\oplus (i+1)}$ will induce the natural map 
			$$\varpi: K_i \longrightarrow F^{\oplus (i+1)}.$$
			From the construction of $\varpi$, it follows that $\pi_* \varpi \, =\, f$. Let $K_{i+1}'$ be the kernel of $\varpi$:
			\begin{equation}
				0 \longrightarrow K_{i+1}' \longrightarrow K_i \longrightarrow F^{\oplus (i+1)} \longrightarrow 0
			\end{equation}
			From the definition of $f$ in $(2)$, we get 
			\begin{equation}
				\pi_* K_{i+1}' \cong (L^{-1})^{\oplus (i+1) d \cdot r} \oplus \big(L^{-1} \oplus L^{-2} \oplus \cdots \oplus L^{(-d+i+1)}\big)^{\oplus d \cdot r}\
			\end{equation}
			Let $K_{i+1}$ be the vector bundle $K_{i+1}' \otimes \pi^* L$. Using the projection formula, we get
			\begin{equation}
				\pi_* K_{i+1} \cong \cO_Y^{\oplus (i+2) d \cdot r} \oplus \big(L^{-1} \oplus L^{-2} \oplus \cdots \oplus L^{(-d+i+2)}\big)^{\oplus d \cdot r}\
			\end{equation}
		\end{enumerate}
		\end{proof}
	
	\textbf{\underline{Step 4}}:\; \textbf{Conclusion} We have proved in Step $1$ and Step $2$ that the assumption of the inductive step satisfies for $i=1, \, 2$.  Hence, by the above inductive argument, we can construct a sequence of vector bundles $\{K_i\}_{i=0}^{(d-2)}$ which fits in the short exact sequence $0 \longrightarrow K_{i+1}' \longrightarrow K_i \longrightarrow F^{\oplus (i+1)} \longrightarrow 0$ such that  
	$$\pi_*K_{i+1}'\, \cong\,(L^{-1})^{\oplus (i+2) d \cdot r} \oplus {\bigg( L^{-2} \oplus \cdots \oplus L^{(-d+i+1)}\bigg)^{\oplus d \cdot r}}, \, \, \text{and} \, \, K_{i+1} :\,=\, K_{i+1}' \otimes \pi^* L, $$
	where the initial vector bundle $K_0\,=\,\cO_X^{\oplus d \cdot r}$.
	Thus we get $\pi_*K_{d-1}'\,\cong\,(L^{-1})^{\oplus d \cdot d \cdot r}$. Let $E$ be the vector bundle $K_{d-1}' \otimes \pi^* L$. Then the vector bundle $E$ of rank $d \cdot r$ has the desired property.
\end{proof}

\begin{remark}
	\label{extensionthm}
	The idea is to reduce the existence of \ul bundles to the existence of relatively \ul bundles for coverings of plane curves. Let $\pi\,:\,X \,\longrightarrow\, \bP^2$ be a cyclic covering of degree $d$. The branch locus $B$ of $\pi$ is a smooth projective plane curve of degree $d \cdot k$ for some $k \in \bN$. We have $\pi_*\cO_X \,\cong\, \cO_{\bP^2} \oplus \cO_{\bP^2}(-k) \oplus \cO_{\bP^2}(-2k) \oplus \cdots \oplus \cO_{\bP^2}(-(d-1)k)$. Let $D$ be a smooth plane curve of degree $k$ that intersects $B$ transversally, and thus the scheme structure on $B \cap D$ is reduced. We consider the following Cartesian diagram:
	\begin{center}
		\begin{tikzcd}
			{{C}} && {{X}} \\
			{D} && {\mathbb{P}^2}
			\arrow[hook, from=2-1, to=2-3]
			\arrow["\pi'", from=1-1, to=2-1]
			\arrow["\pi"', from=1-3, to=2-3]
			\arrow[hook, from=1-1, to=1-3]
		\end{tikzcd}
	\end{center}
The restriction map $\pi': C \longrightarrow D$ is a cyclic covering of degree $d$. Since $B \cap D$ is reduced, it implies $C$ is a smooth projective curve contained in $X$. The branch locus $B'=B \cap D$ of $\pi'$ has $d \cdot k \cdot \text{deg}(D)=dk^2$ many distinct points. Since $H^1(\bP^2, \,\cO_{\bP^2}(j))=0$ for all $j \in \bZ$, the assumptions in Theorem \ref{generalextensionthm} are satisfied. Thus, if $F$ is a relatively \ul bundle on $C$ of rank $r$, then there exists a vector bundle $E$ on $X$ of rank $d \cdot r$ such that $$\pi_*E \,\cong\, \cO_{\bP^2}^{\oplus (d^2 \cdot r)}.$$\
\end{remark}\

We can generalize Theorem \ref{generalextensionthm} for a sequence of complete intersection subvarieties of $\bP^n$. This will provide an alternative approach to prove the existence of an \ul bundle for cyclic coverings of $\bP^n$ as follows.  Let $\pi:X \longrightarrow \bP^n$ be a degree $d$ cyclic covering with branch divisor $B$, which is a smooth hypersurface of degree $d \cdot k$. We will construct certain complete intersection subvarieties $W_{i}$ of $\bP^n$ for $i=0, 1, \cdots, (n-1)$ as follows. We have $\pi_*\cO_X \cong \cO_{\bP^n} \oplus \cO_{\bP^n}(-k) \oplus \cdots \oplus \cO_{\bP^n}(-k(d-1))$. We will assume $W_0=\bP^n$. Let $W_{1}$ be a smooth hypersurface of $\bP^n$ of degree $k$ such that its intersection with the branch divisor $B$ is smooth. Thus the restriction $\pi_1:Z_{1}=\pi^{-1}(W_1) \longrightarrow W_{1}$ is a degree $d$ cyclic covering of smooth varieties. Let $\cO_{W_1}(1)$ be the restriction $\cO_{\bP^{n}}(1)_{|W_{1}}$. Then we will have $\pi_*\cO_{Z_1} \cong \cO_{W_1} \oplus \cO_{W_1}(-k) \oplus \cdots \oplus \cO_{W_1}(-k(d-1))$. Next step is to pick a smooth divisor $W_{2} \in |\cO_{W_1}(k)|$ which intersects the branch divisor of $\pi_1$ smoothly, and we define $Z_2=(\pi_1)^{-1}(W_2)$. By continuing this inductive procedure, we will get a complete intersection curve $W_{n-1}$. The curve $W_{n-1}$ is a complete intersection of type $(k, k, \cdots, k)$. For all $i=0,1, \cdots, (n-1)$ we will have $\text{dim}(W_{i-1})=\text{dim}(W_{i})+1$. As a result of the inductive construction, we will get the following diagram, with each square being a Cartesian

\begin{equation}
\begin{tikzcd}
	{Z_{n-1}} && {Z_{n-2}} && \cdots && {Z_1} && Z_{0}=X \\
	\\
	{W_{n-1}} && {W_{n-2}} && \cdots && {W_1} && W_{0}={\bP^n}
	\arrow[hook, from=3-7, to=3-9]
	\arrow["{\pi_1}", from=1-7, to=3-7]
	\arrow["{\pi_{n-2}}", from=1-3, to=3-3]
	\arrow["{\pi_{n-1}}", from=1-1, to=3-1]
	\arrow[hook, from=3-1, to=3-3]
	\arrow[hook, from=1-1, to=1-3]
	\arrow[hook, from=1-7, to=1-9]
	\arrow["\pi", from=1-9, to=3-9]
	\arrow[dashed, hook, from=1-5, to=1-7]
	\arrow[dashed, hook, from=1-3, to=1-5]
	\arrow[dashed, hook, from=3-3, to=3-5]
	\arrow[dashed, hook, from=3-5, to=3-7]
\end{tikzcd}
\end{equation}
\begin{theorem}\label{completeintthm}
 Let $\pi:X \longrightarrow \bP^n$ be a degree $d$ cyclic covering with branch divisor $B$, which is a smooth hypersurface of degree $d \cdot k$. Let $Z_i$ and $W_i$ are the subvarieties in the above construction, and $\pi_i:Z_i \longrightarrow W_i$ are the degree $d$ cyclic coverings for $i=0, 1, \cdots, (n-1)$. If $F$ is a relatively \ul bundle on $Z_{n-1}$ for $\pi_{n-1}$ of rank $r$, then we can construct an \ul vector bundle $E$ on $X$ of rank $r \cdot d^{(n-1)}$.
\end{theorem}

\begin{proof}
	Using Theorem \ref{generalextensionthm} we can inductively use modifications along $F$ to construct an \ul bundle of the mentioned rank provided the cohomological vanishing assumptions are satisfied for $W_{(n-2)}, W_{(n-3)}, \cdots, W_{0}$, where $\text{dim}\,(W_{(n-2)})=2$. For all $i$ we have 
	$$\pi_*{Z_i} \,\cong\, \cO_{W_i} \oplus \cO_{W_i}(-k) \oplus \cdots \oplus \cO_{W_i}(-k(d-1))$$ where inductively $\cO_{W_{i}}(1)$ is the restriction $\cO_{W_{(i-1)}}(1)_{|W_{i}}$ for all $i$. Thus it is enough to show that $H^1(W_i, \, \cO_{W_i}(j))\,=\,0$ for all $i\,=\,0, 1, \cdots, n-2$, and for all $j \,\in\, \bZ$. When $n\,=\,2$, for $W_0\,=\,\bP^2$ we have the vanishing $H^{1}(\bP^2, \, \cO_{\bP^2}(j))\,=\,0$. So we will assume $n \,\ge\, 3$. We will have the short exact sequence
	$$0 \longrightarrow \cO_{\bP^{n}}(-k) \longrightarrow \cO_{\bP^{n}} \longrightarrow  \cO_{W_1}  \longrightarrow 0$$
	Thus taking long exact sequence in cohomology we get $$H^{1}(W_1, \, \cO_{W_1}(j))\,=\,\cdots\,=\,H^{n-2}(W_1,\,\cO_{W_1}(j))\,=\,0.$$ For the next step, we consider
	$$0 \longrightarrow \cO_{W_1} (-k) \longrightarrow \cO_{W_1}  \longrightarrow  \cO_{W_2}  \longrightarrow 0$$
	Thus taking long exact sequence in cohomology we get $$H^{1}(W_2,\,\cO_{W_2}(j))\,=\,\cdots\,=\,H^{n-3}(W_2,\,\cO_{W_2}(j))\,=\,0.$$ 
	In general, we consider
	$$0 \longrightarrow \cO_{W_i}(-k) \longrightarrow \cO_{W_i}  \longrightarrow  \cO_{W_{i+1}}  \longrightarrow 0$$
	Inductively we are given that $H^{1}(W_i,\,\cO_{W_i}(j))\,=\,\cdots=H^{n-i-1}(W_i,\,\cO_{W_i}(j))\,=\,0$. Thus we will get $H^{1}(W_{i+1},\,\cO_{W_{i+1}}(j))\,=\,\cdots\,=\,H^{n-i-2}(W_{i+1},\,\cO_{W_{i+1}}(j))\,=\,0$. Hence for the term $i+1=n-2$ we will get $$H^{1}(W_{n-2},\,\cO_{W_{(n-2)}}(j))\,=\,0.$$
	\end{proof}

\begin{remark}
	A necessary and sufficient criterion for the existence of relatively Ulrich sheaves for
	an admissible abelian Galois cover $\pi \,:\,X \longrightarrow Y$ of
	projective varieties has been established in \cite[Theorem 1.4]{BiswasPine2026108355}. As a corollary, it is shown that every smooth admissible abelian Galois covering of $\bP^n$ supports an Ulrich bundle; see \cite[Corollary 1.6]{BiswasPine2026108355}.
\end{remark}

\begin{remark}
	 For the existence of an \ul bundle for the cyclic covering $\pi:X \longrightarrow \bP^n$, we need not assume the smoothness of the intermediate coverings due to Proposition \ref{intermediateprop}. 
\end{remark}

\begin{proposition}
	\label{intermediateprop}
	Let $\pi\,:\,X \,\longrightarrow\, Y$ be a degree $d$ cyclic covering of smooth varieties. Let $E$ be a coherent sheaf such that $\pi_*E$ is a vector bundle. Then $E$ is a vector bundle on $X$.
\end{proposition}
\begin{proof}
	Let $x=x_0$ be any point on $X$, and $y=\pi(x)$ be its image on $Y$. Let $\pi^{-1}(y)=\{x=x_0, x_1, x_2, \cdots x_k\}$ for some $k \le (d-1)$. Let $\widehat{\cO_{Y, y}} \longrightarrow \widehat{\cO_{X, x_i}}$ be the finite map of complete local rings induced by the structure map. Then we will have the following isomorphism
	\begin{equation}\label{localeq}
		\widehat{(\pi_*E)_{y}} \cong \oplus_{i=0}^{k} \widehat{E_{x_i}}
	\end{equation}
as an $\widehat{\cO_{Y, y}}$ module. Since $\pi_*E$ is locally free, we have $\text{depth}(\widehat{(\pi_*E)_{y}}) = \text{dim}(Y)=\text{dim}(X)$. From \eqref{localeq} we will get 
\begin{equation}
	\text{depth}(\widehat{(\pi_*E)_{y}}) = \text{min}_{i=0}^{k}\text{depth}(\widehat{E_{x_i}})
\end{equation}  
Since $\widehat{\cO_{Y, y}} \longrightarrow \widehat{\cO_{X, x_i}}$ is a finite map, $\text{depth}(\widehat{E_{x_i}})$ as an $\widehat{\cO_{X, x_i}}$ is same as when considered as an $\widehat{\cO_{Y, y}}$  module. Thus we will have  $$\text{depth}_{\widehat{\cO_{X, x_i}}}(\widehat{E_{x_i}}) \ge \text{min}_{i=0}^{k}\text{depth}(\widehat{E_{x_i}}) = \text{depth}(\widehat{(\pi_*E)_{y}})=\text{dim}(X)$$
Since $\text{depth}_{\widehat{\cO_{X, x_i}}}(\widehat{E_{x_i}})$ is always $\le \text{dim}(X)$, we will have $\text{depth}_{\widehat{\cO_{X, x_i}}}(\widehat{E_{x_i}}) = \text{dim}(X)$ for all $i=0, 1, \cdots, k$. In particular for any $x \in X$, we will get $\text{depth}_{\widehat{\cO_{X, x}}}(\widehat{E_{x}}) = \text{dim}(X)$. Since $X$ is smooth, $\text{p.d}(\widehat{E_{x}})$ is finite. From Auslander-Buchsbaum theorem we get $\text{p.d}(\widehat{E_{x}})=0$. Hence $E$ is a vector bundle.
\end{proof}

\begin{remark}
For Proposition \ref{intermediateprop}, we have assumed that $\pi$ is a finite cyclic covering. But the proof still works if $\pi$ is only assumed to be a finite covering.
\end{remark}

\section{\ul bundles on cyclic coverings of $\bP^2$}
\label{five}
In this section, we consider the more specific case of cyclic coverings $X$ of the projective plane $\bP^2$ of arbitrary degree $d$. In \S 3  we have proved the existence of \ul bundle on $X$ by methods involving Veronese embedding, and proving the existence of \ul sheaf on certain hypersurface using matrix factorization of polynomials. By using methods developed in \S 4, in this section we will give a different proof of the existence of \ul bundle by reducing the problem to showing the existence of relatively \ul bundles on the covering of plane curves. We will set up the following notations.\\

Let $\pi:X \longrightarrow \bP^2$ be a cyclic covering of degree $d$. The branch locus $B$ of $\pi$ is a smooth projective plane curve of degree $d \cdot k$ for some $k \in \bN$. Let $D$ be a smooth plane curve of degree $k$ which intersects $B$ transversally, and thus the scheme structure on $B \cap D$ is reduced. We consider the following Cartesian diagram
\begin{center}
	\begin{tikzcd}
		{{C}} && {{X}} \\
		{D} && {\mathbb{P}^2}
		\arrow[hook, from=2-1, to=2-3]
		\arrow["\pi'", from=1-1, to=2-1]
		\arrow["\pi"', from=1-3, to=2-3]
		\arrow[hook, from=1-1, to=1-3]
	\end{tikzcd}
\end{center}

The restriction map $\pi': C \longrightarrow D$ is a cyclic covering of degree $d$. Since $B \cap D$ is reduced, $C$ is a smooth projective curve contained in $X$. The branch locus $B'=B \cap D$ of $\pi'$ has $d \cdot k \cdot \text{deg}(D)=d \cdot k^2$ many distinct points.\\

Let $d_1 \le d_2 < d$ be three natural numbers. Let $X_1\, =\, $ $\{F \in H^0(\bP^2,\,\cO_{\bP^2}(d)): F\,=\,F_1 \cdot G_1, \text{deg}(F_1)\,=\,d_1\}$ and let $X_2\, =\, $  $\{F \in H^0(\bP^2,\,\cO_{\bP^2}(d)):\, F\,=\,F_2 \cdot G_2, \text{deg}(F_2)\,=\,d_2\}$. Then the main result of \cite[Theorem 5.1]{carlini2008complete} applied to $\bP^2$ states that the join of $X_1$ and $X_2$ is $\bP(H^0(\bP^2,\,\cO_{\bP^2}(d)))$. In other words, every degree $d$ form $F$ can be written as $F\,=\,F_1G_1\,+\,F_2G_2$, where $F_1$ is of degree $d_1$, and $F_2$ is of degree $d_2$. Further, we prove the following

\begin{proposition}
\label{smoothnesstrans}
A generic degree $d$ form can be written as $F\,=\,F_1G_1\,+\,F_2G_2$ such that $F_1$ defines a smooth curve, and $F_1$ is transversal to both $F_2$ and $G_2$. 
\end{proposition}
\begin{proof}
 Let $U$ be the open set containing smooth degree $d_1$ forms in $\bP H^0(\bP^2,\,\cO_{\bP^2}(d_1))$. Let $\mathcal{U}$ be the following dense open subset of $U \times \bP H^0(\bP^2,\,\cO_{\bP^2}(d_2) \times \bP H^0(\bP^2,\,\cO_{\bP^2}(d-d_2))$
 $$\cU \,=\,  \{(t,\,s,\,w)\,:\, Z(F_t) \bigcap Z(G_{s}) \bigcap Z(H_{w}) = \emptyset \}$$
 where the triple $(t,s,w)$ parameterizes degree $d_{1}$ form $F_{t}$, degree $d_{2}$ form $G_{s}$, and degree $(d-d_{2})$ form $H_{w}$ respectively. The openness of $\cU$ can be seen as follows. For any generic pair $(t, \, s)$, the forms $F_t$ and $G_s$ intersect in finitely many points. Then a generic $H_w$ will avoid these finitely many points. Similar arguments for any generic pair $(t, \, w)$ and $(s, \, w)$. 
 
 Let $\cC$ be the following set 
 $$\cC=\{(t,s,w)\in \cU: F_t \; \text{is not transversal to} \; G_s \cdot H_w  \} \subseteq U \times \bP H^0(\bP^2,\,\cO_{\bP^2}(d_2)) \times \bP H^0(\bP^2,\,\cO_{\bP^2}(d-d_2))$$
 Let $V$ and $W$ be the following subsets:
 $$V=\{(t,\,s,\, w)\,\in\,\cU\,:\, F_t \; \text{is not transversal to}\; G_s\},$$  $$W= \{(t,\, s,\,w)\,\in \cU: F_t \; \text{is not transversal to} \; H_w\}.$$
 Then we will have the following
 $$\cC \,=\,  V \,\bigcup\, W.$$
 
Let us assume that both $V$ and $W$ are proper closed subschemes of $\cU$. Then we consider the map $$(\cU \setminus \cC) \times \bP (H^0(\bP^2,\,\cO_{\bP^2}(d-d_1))) \longrightarrow \bP (H^0(\bP^2,\,\cO_{\bP^2}(d))).$$
 
Since $\cC$ is a proper closed subscheme of $\cU$, the complement $(\cU \setminus \cC)$ is an open dense subset of $\bP (H^0(\bP^2,\,\cO_{\bP^2}(d_1))) \times \bP (H^0(\bP^2,\,\cO_{\bP^2}(d_2))) \times \bP (H^0(\bP^2,\,\cO_{\bP^2}(d-d_2)))$. Hence by \cite[Theorem 5.1]{carlini2008complete}, it follows that the image is an open subset of $\bP(H^0(\bP^2,\,\cO_{\bP^2}(d)))$.\\

Now we will prove $V$ is a proper closed subscheme of $U \times \bP H^0(\bP^2,\,\cO_{\bP^2}(d_2)) \times \bP H^0(\bP^2,\,\cO_{\bP^2}(d-d_2))$. We consider $$V_1:=\{(t,s): F_t \; \text{is not transversal to} \; G_s\} \subseteq U \times \bP (H^0(\bP^2,\,\cO_{\bP^2}(d_2))).$$
It follows that the subscheme $V$ is the pullback of the subscheme $V_1$ under the projection $p_{12}:\, U \times \bP H^0(\bP^2,\,\cO_{\bP^2}(d_2) \times \bP H^0(\bP^2,\,\cO_{\bP^2}(d-d_2)) \longrightarrow \, U \times \bP H^0(\bP^2,\,\cO_{\bP^2}(d_2)$.  Thus it is enough to show that $V_1$ is a proper closed subscheme of $U \times \bP (H^0(\bP^2,\,\cO_{\bP^2}(d_2)))$.

 Let $(N+1)$ be the dimension of $H^{0}(\bP^2,\,\cO_{\bP^2}(d_{2}))$.  Let $\phi_{d_{2}}: \bP^2 \longrightarrow \bP(H^{0}(\bP^2,\,\cO_{\bP^2}(d_{2}))) \cong \bP^N$ be the degree $d_{2}$ Veronese embedding.  Any homogeneous polynomial of degree $d_2$ can be rewritten as  a linear homogeneous polynomial in the Veronese coordinates of $\bP^N$. Thus degree $d_{2}$ forms $G_s$ will be in bijective correspondence with hyperplanes in $\bP^N$ i.e., elements in $(\bP^N)^{\vee}$. Let $H_{G_{s}}$ be the associated hyperplane in $\bP^N$. Let $\phi_{d_{2}}(F_{t})$ be the image of the plane curve $F_{t}$ under the Veronese embedding $\phi_{d_{2}}$. The dual variety $\phi_{d_{2}}(F_t)^{\vee} \subseteq (\bP^N)^{\vee}$ is a hypersurface unless $\phi_{d_{2}}(F_t)$ is linear. This can be seen as follows. All the hyperplanes of $\bP^N$ which are tangent to $\phi_{d_{2}}(F_t)$ form the dual variety $\phi_{d_{2}}(F_t)^{\vee}$ (for details on the dual variety, refer to \cite[\S 1 and \, \S 2]{LAMOTKE198115}). At each point of the curve $\phi_{d_{2}}(F_t)$, the space of hyperplanes which are tangent to that point can be identified as $\bP^{(N-2)}$. When $\phi_{d_{2}}(F_t)$ is linear, one gets a constant family of the planes $\bP^{(N-2)}$. Thus, the dimension of the dual variety is $N-2$. When $\phi_{d_{2}}(F_t)$ is not linear, there is a change of tangent direction when passing from one point to another. Thus one gets a non trivial family of  $\bP^{(N-2)}$ over the curve $\phi_{d_{2}}(F_t)$. Hence,  $\phi_{d_{2}}(F_t)^{\vee}$ is a hypersurface. 
 
 If $\phi_{d_{2}}(F_t)$ is a linear space, then $d_{2}=1$, and $d_{1}=1$. In this case $V_1$ is the diagonal of $\bP(H^{0}(\bP^2,\,\cO_{\bP^2}(1))) \times \bP(H^{0}(\bP^2,\,\cO_{\bP^2}(1)))$. Hence, it follows that $V_1$ is a proper closed subscheme. For the rest, we assume $\phi_{d_{2}}(F_t)$  is not linear. The degree of the hypersurface $\phi_{d_{2}}(F_t)^{\vee}$ is determined by the degree of the curve $\phi_{d_{2}}(F_{t})$ which is $d_{1} \cdot d_{2}$ and the genus of the curve $\binom{d_{1}-1}{2}$. This can be seen as follows. For a generic line $L \subset (\bP^N)^{\vee}$, one obtains a Lefschetz pencil $\phi_{d_{2}}(F_t) \longrightarrow \bP^1$. Thus, the ramification is allowed to have at most one double point on each fiber. Hence, it is a simple covering and the degree of $\phi_{d_{2}}(F_t)^{\vee}$ is precisely the ramification divisor which can be computed by the Riemann-Hurwitz formula. For a smooth form $F_{t} \in U$, let the degree of the hypersurface $\phi_{d_{2}}(F_t)^{\vee}$ be $k$. Then $F_t \; \text{is not transversal to} \; G_s$ if and only if $H_{G_s}$ is an element of the dual variety $\phi_{d_2}(F_t)^{\vee} \subseteq (\bP^N)^{\vee}$. The space $\bP(H^0((\bP^N)^{\vee},\,\cO_{(\bP^N)^{\vee}}(k)))$ be the space of degree $k$ hypersurface in $(\bP^N)^{\vee}$. The family $\{F_{t}:t \in U\}$ will define an injective map $U \longrightarrow \bP(H^0((\bP^N)^{\vee},\,\cO_{(\bP^N)^{\vee}}(k)))$ by $t \mapsto \phi_{d_2}(F_t)^{\vee}$. The injectivity follows from the isomorphism $(\phi_{d_2}(F_t)^{\vee})^{\vee} \cong \phi_{d_{2}}(F_t)$ \cite[2.2. Duality Theorem]{LAMOTKE198115}.  We consider the incidence scheme  

$$I\,=\,\{(X,\, t) \,:\, t \,\in\, X\} \,\subseteq\, \bP(H^0((\bP^N)^{\vee}, \, \cO_{(\bP^N)^{\vee}}(k))) \times (\bP^N)^{\vee}.$$

The scheme $U \times (\bP^N)^{\vee}$ can not be contained inside $I$ since there are transversally intersecting pairs $F_{t}$, and $G_{s}$. Since $U \times (\bP^N)^{\vee}$ is irreducible,  $I \cap (U \times (\bP^N)^{\vee})$ is a proper closed subscheme of $U \times (\bP^N)^{\vee}$. From the construction, it follows that  $I \cap (U \times (\bP^N)^{\vee}) =V_1$. Thus $V_1$ is a proper closed subscheme of $U \times \bP (H^0(\bP^2,\,\cO_{\bP^2}(d_2))$. Hence $V$ is a proper closed subscheme of $\cU$.

Analogously, we define $$W_1: \, =\{(t,w): F_t \; \text{is not transversal to} \; H_w\} \subseteq U \times \bP (H^0(\bP^2,\,\cO_{\bP^2}(d-d_2))).$$
Then it follows that $W$ is the pullback of the subscheme $W_1$ under the projection $p_{13}:\, U \times \bP H^0(\bP^2,\,\cO_{\bP^2}(d_2)) \times \bP H^0(\bP^2,\,\cO_{\bP^2}(d-d_2)) \longrightarrow \, U \times \bP H^0(\bP^2,\,\cO_{\bP^2}(d-d_2))$. Similar arguments will show that $W_1$ is a proper closed subscheme of $U \times \bP (H^0(\bP^2,\,\cO_{\bP^2}(d-d_2))$. Hence, we will have $W$ is a proper closed subscheme of $\cU$.
 
\end{proof}

\begin{proposition}
	\label{linelemma}
	Let $f\,:\,\bP^1 \longrightarrow \bP^1$ be a degree $d$ finite covering. Then $\cO_{\bP^{1}}(d-1)$ is an \ul line bundle.
\end{proposition}

\begin{proof}
	We claim that $f_*\cO_{\bP^1}(d-1) \,\cong\, \cO_{\bP^1}^{\oplus d}$. The rank $d$ vector bundle $f_*\cO_{\bP^1}(d-1)$ can be written as $\oplus_{i=1}^d \mathcal{L}_i$. By Riemann-Roch, we have $\text{deg}\,(f_*\cO_{\bP^1}(d-1))=0$, and hence $\sum_{i=1}^d\text{deg}\,(\cL_i)=0$.  Let deg\,($\mathcal{L}_i) \,>\,0$ for some $i$. By projection formula we have $f_*(\cO_{\bP^1}(d-1) \,\otimes\, f^*(\cO_{\bP^1}(-1)) \,\cong\, \oplus_{i=1}^d \mathcal{L}_i(-1)$. Thus $H^0(\bP^1, \, f_*(\cO_{\bP^1}(d-1) \otimes f^*(\cO_{\bP^1}(-1))) \ne 0$ since degree$(\mathcal{L}_i(-1)) \ge 0$ for some $i$. On the other hand, since $\text{deg}\,(f)\,=\,d$, we have $f^*(\cO_{\bP^1}(-1) \,=\, \cO_{\bP^1}(-d)$. Thus $H^0(\bP^1, \, f_*(\cO_{\bP^1}(d-1) \otimes f^*(\cO_{\bP^1}(-1)))\,=\,H^0(\bP^1, \,f_*(\cO_{\bP^1}(d-1-d)))\,=\,H^0(\bP^1, \, f_*(\cO_{\bP^1}(-1)))\,=\, H^0(\bP^1, \cO_{\bP^1}(-1))=0$. This is a contradiction. Thus $f_*\cO_{\bP^1}(d-1) \,\cong\, \cO_{\bP^1}^{\oplus d}$.

\end{proof}

The following is the main theorem of this section
\begin{theorem}\label{evenparitythm}
	Let $\pi:X \longrightarrow \bP^2$ be a generic cyclic covering of degree $d$ such that the degree of the branch divisor $d \cdot k$ is even. Then there exists an \ul bundle $E$ on $X$ of rank $d$. 
\end{theorem}

\begin{proof}
 Let $\frac{d \cdot k}{2}$ be the natural number $c$. The branch divisor $B$ defined by a polynomial $F$ of degree $ d \cdot k$ can be written as $F\,=\,F_1G_1+F_2G_2$, where $F_1$ is of degree $k$, and $F_2$ is of degree $c$, by \cite{carlini2008complete}[Theorem 5.1]. By Proposition \ref{smoothnesstrans}, those cyclic coverings for which $F_1$ can be chosen to be smooth, and $F_1$ intersects $F_2 \cdot G_2$ transversally, form an open dense subset of the space of all cyclic coverings such that the branch is of degree $ d \cdot k$. So for generic cyclic coverings $\pi$, we can assume $F_1$ is smooth, and $F_1$ intersects $F$ transversally. Let $D$ be the smooth plane curve of degree $k$ defined by $F_1$. We consider the following Cartesian diagram
\begin{center}
	\begin{tikzcd}
		{{C}} && {{X}} \\
		{D} && {\mathbb{P}^2}
		\arrow[hook, from=2-1, to=2-3]
		\arrow["\pi'", from=1-1, to=2-1]
		\arrow["\pi"', from=1-3, to=2-3]
		\arrow[hook, from=1-1, to=1-3]
	\end{tikzcd}
\end{center}
The branch locus of $\pi'$ is $B'$ and is given by $D \cap Z(F_2\cdot G_2)$, which is $dk^2$ distinct points since the intersection is transversal. Let $M$ be the line bundle $\cO_{\bP^2}(k)_{|D}$. Thus $B'$ is in $|M^{\otimes d}|$. The restrictions of $F_2$, and $G_2$ denoted by $F_2'$ and $G_2'$ are sections of $\cO_{\bP^2}(c)_{|D}$. Since $B'$ is a reduced scheme of distinct points, the sections $F_2'$ and $G_2'$ are linearly independent. Also, they can not simultaneously vanish at a point of $D$. We define the following map $h$ given by $F_2'$, and $G_2'$ 
\begin{center}
	\[\begin{tikzcd}
		& \bP^1 \\
		D & {\bP^1}
		\arrow["f", from=1-2, to=2-2]
		\arrow["h={|F_2',G_2'|}"', from=2-1, to=2-2]
	\end{tikzcd}\]
\end{center} 
where $f:\bP^1 \longrightarrow \bP^1$ is defined by $x \mapsto x^d, y \mapsto y^d$. Using Riemann-Hurwitz formula the ramification divisor of $f$ is given by $x^{d-1}y^{d-1}$ i.e., $(d-1)\{0\} \cup (d-1) \{\infty\}$. The algebra structure of $f$ is $f^{\#}:\bC[x, y] \longrightarrow \bC[x,y]$ defined by $x \mapsto x^d$, and $y \mapsto y^d$. Thus, the scheme theoretic image of the branch locus is given by the ideal $I=(f^{\#})^{-1}(x^{d-1}y^{d-1})=(xy)$. Thus the branch locus of $f$ is $\{0, \infty\}$. The map $h$ is of degree $\frac{d \cdot k^2}{2}$. Since the branch $B'$ contains distinct points, the map $h$ is etale at $\{0, \infty\}$. The pullback $h^{-1}\{0, \infty\}$ is the branch $B'$ of $\pi'$. We consider the following Cartesian diagram 
\begin{equation}
	\label{commutitivity}
	\begin{tikzcd}
		C' & \bP^1 \\
		D & {\bP^1}
		\arrow["h'", from=1-1, to=1-2]
		\arrow["{f'}"', from=1-1, to=2-1]
		\arrow["f", from=1-2, to=2-2]
		\arrow["h", from=2-1, to=2-2]
	\end{tikzcd}
\end{equation}
Since the branch divisor of $f':C' \longrightarrow D$ is $B'$, from the uniqueness of the cyclic covering of $D$ for a given branch divisor, we get $C' = C$, and $f'=\pi'$.
 
By Proposition \ref{linelemma} we have $f_*\cO_{\bP^1}(d-1) \cong \cO_{\bP^1}^{\oplus d}$. By the base change theorem \cite[\href{https://stacks.math.columbia.edu/tag/02KG}{Tag 02KG}]{stacks-project} it follows $h'^*(\cO_{\bP^1}(d-1))$ is a relatively \ul line bundle on $C$ for $\pi'$. Hence by Theorem \ref{generalextensionthm}, there exists an \ul bundle $E$ on $X$ of rank $d$. 
\end{proof}

\subsection{Estimation of the rank of Ulrich bundle when $d \cdot k$ is odd} Let $\pi: X \longrightarrow \bP^2$ be a cyclic covering of degree $d$ with branch divisor $B$ a curve defined by a polynomial $F$ of degree $d \cdot k$ such that $d \cdot k$ is odd. We will develop an algorithm to give an estimation of the rank of the Ulrich bundle on $X$. 

When $ d \cdot k$ is odd, we choose the smallest prime $p$ that divides $d \cdot k$. Thus, we have $d \cdot k\,=\,p \cdot r$ for some natural number $r \in \bN$. The homogeneous polynomial $F$ of degree $d \cdot k$ in $\bC[x, y, z]$ defining the branch divisor $B$ can be written as 
\begin{equation} \label{crudecom}
	F \,=\,F_1 \cdot G_1\,+\,F_2 \cdot G_2,
\end{equation}
 such that $F_1$ is of degree $k$ and $F_2$ is of degree $r$. We are interested in estimating the size $m$ of a matrix factorization of the homogeneous polynomial $F_2 \cdot G_2$ in the following form 
\begin{equation} \label{newmfdefn}
	A_1 \cdot A_2 \cdot \cdots \cdot A_{p}\,=\, F_2 \cdot G_2 \cdot \id_{m \times m},
\end{equation}
such that $A_i$ are matrices of size $m$ with entries of each $A_i$ being degree $r$ homogeneous polynomials in $\bC[x, y, z]$. An estimation of size $m$ is used in Theorem \ref{dkodd}.

\begin{remark} \label{keymfremark3}
	Note that the above interpretation of matrix factorization \eqref{newmfdefn} is different  from that in Definition \ref{MF}. In Definition \ref{MF}, the entries of the matrices are required to be linear homogeneous polynomials. With respect to a degree $r$ Veronese embedding of $\bP^2 \hookrightarrow \bP^N$, the same matrices $A_i$ can be written as matrices whose entries are linear homogeneous polynomials in the Veronese coordinates of $\bP^N$. Thus, for our purpose, the above notion \eqref{newmfdefn} suffices (for more details, refer to the proof of Theorem \ref{dkodd}).
\end{remark}

We have described a recursive method to construct a matrix factorization of the homogeneous polynomial $F_2G_2$ of degree $d \cdot k$ appearing in \eqref{crudecom}. The main recursive step is to express $F_2G_2$ as a sum of products of homogeneous polynomials such that the degrees of these polynomials are of the form $q \cdot r$, where $q< p$ is a prime (possibly depending on the term). Choosing the smallest prime $p$ such that $d \cdot k=p \cdot r$ is a way to achieve a matrix factorization of $F_2G_2$ in a possibly minimal number of steps, as described in Proposition \ref{easycase}, Corollary \ref{oddcase}, and in \S \ref{difficultcase}. In a sense, choosing the smallest prime dividing $d \cdot k$ is also used in Theorem \ref{evenparitythm}, when $d \cdot k$ is even. We crucially used that $2$ divides $ d \cdot k$ in the proof of Theorem \ref{evenparitythm}.

\begin{remark} \label{keymfremark2}
	Let $f$ be a homogeneous polynomial of degree $r$. Let $\{A_i\}_{i=1}^{p}$ be a matrix factorization of a degree $p \cdot r$ polynomial $Q$ of size $m$ as in \eqref{newmfdefn}. Then the product  $f \cdot Q$ also admits a matrix factorization of the same size $m$ as follows
	$$ \begin{bmatrix}
		f & 0 & \cdots & 0 \\
		0 & f & \cdots & 0 \\
		\cdots & \cdots & \cdots  & \cdots \\
		0 & 0 & \cdots & f 
	\end{bmatrix}  \cdot A_1 \cdot A_2 \cdots A_{p}\,=\, f \cdot \id_{m \times m} \cdot Q \cdot \id_{m \times m}=(f \cdot Q) \id_{m \times m}
	.$$
\end{remark}

\vspace{20pt}

We consider the following function
\begin{eqnarray}
	\sigma\,:\,\bN_{\ge 2} &\longrightarrow& \bN_{\ge 1}\\
	\sigma(n) &=& \frac{n-1}{2} \,,\,\text{for odd}\, \,\,n \nonumber\\
	&=&\frac{n}{2} \,,\,\,\text{for even}\,\,\, n\ \nonumber.
\end{eqnarray}
We define $\sigma^{\ell}(n)\,=\,(\sigma \circ \sigma \circ \cdots \circ \sigma)(n)$, where $(\sigma \circ \sigma \circ \cdots \circ \sigma)$ is the $\ell$-fold composition of $\sigma$.
For any $n \in \bN_{\ge 2}$, we can associate a finite sequence $\{\sigma(n),\,\sigma^2(n), \cdots, \sigma^k(n)\}$, where $\sigma^k(n)=1$ for some $k \in \bN$ (the natural number $k$ depends on $n$). We can also see that $\sigma^{k-1}(n)$ is either $2$ or $3$.

Let $p$ be a prime number. We call the prime $p$ an {\it admissible prime} if all the natural numbers $\sigma(p), \sigma^2(p), \cdots, \sigma^{k-1}(p)$ are prime, where $\sigma^k(p)=1$.

The only admissible primes are $2, 3, 5, 7, 11, 23, 47$.\\

We consider a homogeneous polynomial $F$ in $\bC[x, y, z]$ of degree $d \cdot k$. We choose the smallest prime $p$ that divides $d \cdot k$. Thus, we have $d \cdot k\,=\,p \cdot r$ for some $r \in \bN$. 
%Recall that by \cite[Theorem 5.1]{carlini2008complete}, we have  $$F=F_1G_1+F_2G_2,$$ where $F_1$ is a homogeneous polynomial of degree $k$, and $F_2$ is a homogeneous polynomial of degree $r$. 
We provide a method to obtain a matrix factorization as in \eqref{newmfdefn} of the polynomial $F$. The size of this matrix factorization depends on the prime $p$. We will denote the size of this matrix factorization by $m_p$. This estimation $m_p$ is used in Corollary \ref{oddcase}.

%We will denote the size of the matrix factorization of $F$ by $m_{\mu}$. Let $\{p_n \,:\, n \in \bN\}$ be the set of primes. We will denote $p_0=1$.  We denote $\varpi_{n_1}\,=\,\frac{p_n-1}{2}$. When $\frac{\varpi_{n_{i-1}}-1}{2}$ is a natural number for $i \ge 2$, iteratively we define, $\varpi_{n_i}\,=\,\frac{\varpi_{n_{i-1}}-1}{2}$ for $i \ge 2$. When $\varpi_{n_i}$ is a prime number, then by an abuse of notation, we will use $p_{n_i}\,=\,\varpi_{n_i}$. 

In the following, we consider the special case when $p$ is an admissible prime.
\begin{prop} \label{easycase}
	Let $p$ be an admissible prime. We provide a method to obtain a matrix factorization of $F$ such that the size $m_p$ satisfies the recursive formula $$m_{p}=p(p-1)^2(m_{\sigma(p)})^4,$$ with initial values $m_{2}=2$ and  $m_{3}=12$.
\end{prop}

\begin{proof}
	By \cite[Theorem 5.1]{carlini2008complete} $F$ can be written as 
	\begin{equation} \label{keydecom}
	F=F_1G_1+F_2G_2,
	\end{equation}
	 where both $F_1$ and $F_2$ are of degree $r$. Thus, $G_1$ and $G_2$ are of degree $(p-1)r$.
	
	 \textbf{\underline{p=2.}} Both $G_1$ and $G_2$ are of degree $r$. Since both $F_1G_1$ and $F_2G_2$ admit a matrix factorization of size $1$, $F$ will admit a matrix factorization of size $2$ by Remark \ref{keymfremark}. The matrix factorization can also be explicitly written as
	$$ \begin{bmatrix}
		F_1 & F_2 \\
		G_2 & -G_1 
	\end{bmatrix}  \cdot  \begin{bmatrix}
		G_1 & F_2 \\
		G_2 & -F_1 
	\end{bmatrix} \, = \, \begin{bmatrix}
		F_1G_1+F_2G_2 & 0 \\
		0 & F_1G_1+F_2G_2 
	\end{bmatrix} \, = \, F \cdot \id_{2 \times 2}
	.$$
	
	\textbf{\underline{p=3.}} Thus, both $G_1$ and $G_2$ are of degree $2r$. By the $p=2$ case, both $G_1$ and $G_2$ admit a matrix factorization of size $2$. By Remark \ref{keymfremark2}, both $F_1G_1$ and $F_2G_2$ admit a matrix factorization of size $2$. Hence, by Remark \ref{keymfremark}, $F$ admits a matrix factorization of size $3 \cdot 2 \cdot 2=12$ (even though $\deg (F)=\deg (F_1G_1)=\deg (F_2G_2)=3r$, the degree counted for this matrix factorization is $3$, since, in terms of the degree $r$ Veronese coordinates, both $F_1G_1$ and $F_2G_2$ have degree 3, as explained in Remark \ref{keymfremark3}. Hence, by Remark \ref{keymfremark}, the size of the matrix factorization of $F$ is $3 \cdot 2 \cdot 2=12$).
	
	\textbf{\underline{General case.}} Both $G_1$ and $G_2$ are of degree $(p-1)r$. Using \cite[Theorem 5.1]{carlini2008complete}, the polynomials $G_i$ can be written as
	\begin{equation} \label{a}
		G_i\,=\,P_i Q_i + P_i' Q_i',
	\end{equation}
	for $i=1, 2$ such that $P_i , \, Q_i, \, P_i',\, Q_i'$ are of degree $\sigma(p) \cdot r$. Since $p$ is admissible, $\sigma(p)$ is a prime number. By iterative property, each of $P_i , \, Q_i, \, P_i',\, Q_i'$ has a matrix factorization of size $m_{\sigma(p)}$. Thus, using similar arguments as in Remark \ref{keymfremark2}, the products $P_i \cdot Q_i $ and $P_i' \cdot Q_i'$ have a matrix factorization of size $m_{\sigma(p)}$. Hence, by Remark \ref{keymfremark}, $G_i$ has  a matrix factorization of size $(p-1) (m_{\sigma(p)})^2$ for $i=1, 2$. Hence, by Remark \ref{keymfremark2}, the products $F_1G_1$ and $F_2G_2$ will have a matrix factorization of size $(p-1) (m_{\sigma(p)})^2$. Thus, again using Remark \ref{keymfremark}, we conclude $F$ has a matrix factorization of size $p \cdot \big((p-1) (m_{\sigma(p)})^2\big)^2=p \cdot (p-1)^2 \cdot (m_{\sigma(p)})^4$.
	
\end{proof}

\begin{notn} \label{notamf}
	Let $F$ be a homogeneous polynomial of degree $d \cdot k$ in $\bC[x, y, z]$. Let $p$ be the smallest prime that divides $d \cdot k$, i.e., $d \cdot k=p \cdot r$ for some natural number $r \in \bN$. By \cite[Theorem 5.1]{carlini2008complete}, the polynomial $F$ can be written as  $$F=F_1G_1+F_2G_2,$$ where $F_1$ is a homogeneous polynomial of degree $k$, and $F_2$ is a homogeneous polynomial of degree $r$.  We describe a method similar to the proof of Proposition \ref{easycase} to obtain a matrix factorization as in \eqref{newmfdefn} of the polynomial $F_2G_2$.  Recall that the entries of each matrix appearing in this matrix factorization are degree $r$ homogeneous polynomials. It is enough to construct a matrix factorization of $G_2$ (which provides a matrix factorization of $F_2G_2$ of the same size by Remark \ref{keymfremark2}).  The size of this particular matrix factorization of $F_2G_2$ depends on the prime $p$. We will denote the size of this matrix factorization by $M_p$. This estimation $M_p$ is used in Theorem \ref{dkodd}.
	
\end{notn}

	  %In the following, we will describe a process for obtaining a matrix factorization of $F_2 \cdot G_2$ as in Definition \ref{newmfdefn}. We will denote the size of this particular matrix factorization of the homogeneous polynomial $F_2G_2$ by $M_p$. 

%We consider a homogeneous polynomial $F$ in $\bC[x, y, z]$ of degree $p \cdot r$ for some prime $p$. Recall that by \cite[Theorem 5.1]{carlini2008complete}, we have  $$F=F_1G_1+F_2G_2,$$ where $F_1$ is a homogeneous polynomial of degree $k$, and $F_2$ is a homogeneous polynomial of degree $r$. 

%Following the proof of Proposition \ref{easycase}, we get the following
We use the estimate $m_p$ from Proposition \ref{easycase} in the following.
\begin{cor}
	\label{oddcase}
	Let $p$ be an admissible prime. We obtain a matrix factorization of $F_2 \cdot G_2$ such that the size $M_p$ satisfies the recursive formula $$M_{p}\,=\,(p-1)(m_{\sigma(p)})^2,$$ with initial values $M_{2}=1$ and $M_{3}=2$.
\end{cor}

\begin{proof}
	Let $p=2$. Then the degrees of both $F_2$ and $G_2$ are $r$. Thus, we have $(F_2)_{1 \times 1} \cdot (G_2)_{1 \times 1}\,=\,(F_2G_2) \cdot \id_{1 \times 1}$, giving a matrix factorization of size $1$ of $F_2G_2$ as in \eqref{newmfdefn}. Thus, 	$M_{2}=1$. 
	
	Let $p=3$. The degree of $F_2$ is $r$ and the degree of $G_2$ is $2r$. By Proposition \ref{easycase}, $m_{2}=2$, and thus, $G_2$ has matrix factorization of size $m_{2}=2$. Hence, $F_2G_2$ has a matrix factorization of size $M_{3}=2$ by Remark \ref{keymfremark2}.
	
	 In general, the degree of $F_2$ is $r$ and the degree of $G_2$ is $(p-1) \cdot r$. Thus, we have a decomposition $G_2=PQ+P'Q'$, where each of  $P,  Q, P', Q'$ has degree $\sigma(p) \cdot r$. Hence, each admits a matrix factorization of size $m_{\sigma(p)}$, given by the recursive formula in Proposition \ref{easycase}. Thus, by Remark \ref{keymfremark}, $G_2$ has a matrix factorization of size $(p-1)(m_{\sigma(p)})^2$. Hence, by Remark \ref{keymfremark2}, $F_2G_2$ has a matrix factorization of the same size $M_{p}\,=\,(p-1)(m_{\sigma(p)})^2$.
\end{proof}

\begin{example} \label{em}
	In the proof of Corollary \ref{oddcase}, we have seen $M_2=1, \, M_3=2$. By Corollary \ref{oddcase}, we have $M_5\,=\,4 \cdot (m_2)^2\,=\,4 \cdot 4\,=\,16$ by Proposition \ref{easycase}.   
	
	For $p=7$, we have $M_7=6 \cdot (m_3)^2=6 \cdot 144=864$. We have $M_{11}=10 \cdot (m_5)^2=10 \cdot (5 \cdot 16 \cdot (m_2)^4)=12800$. 
\end{example}

\subsubsection{General method for arbitrary prime $p$}
 \label{difficultcase}

	 Let $F$ be an arbitrary homogeneous polynomial in $\bC[x, y, z]$ of degree $d \cdot k$. Then by \cite[Theorem 5.1]{carlini2008complete} $F$ can be written as 
	 \begin{equation} \label{gendecom}
	 	F=F_1G_1+F_2G_2,
	 \end{equation}
	  where $F_1$ is a homogeneous polynomial of degree $k$, and $F_2$ is a homogeneous polynomial of degree $r$. Let $p$ be the smallest prime that divides $d \cdot k$. We have $d \cdot k=p \cdot r$ for some $r \in \bN$. Since $p$ is an arbitrary prime, $\sigma(p)$ need not be prime. We  provide a recursive method to obtain a matrix factorization of $F_2 \cdot G_2$ in the sense of \eqref{newmfdefn}. By Remark \ref{keymfremark2}, it is sufficient to obtain a matrix factorization of $G_2$. We use the size of this matrix factorization of $F_2G_2$, denoted by $M_p$, in Theorem \ref{dkodd}. The key recursive step is to express $F_2G_2$ as a sum of products of homogeneous polynomials such that the degrees of these polynomials are of the form $q' \cdot r$, where $q'< p$ is a prime (possibly depending on the term). If the sizes of the matrix factorizations of these polynomials of degrees $q' \cdot r$, for primes $q'<p$, are known, then this provides a recursive method to compute the size $M_p$. We describe this method with several examples.
	
	  Let $p=13$. We consider $F_2 \cdot G_2$ appearing in \eqref{gendecom}, where $F_2$ has degree $r$ and $G_2$ has degree $12r$. We have $\sigma(13)=6$, which is not prime. The following steps are illustrated using the diagram below. In all such diagrams, we denote a polynomial expression of the form $PQ+P'Q'$ by $(d_1 \cdot d_2+d_1' \cdot d_2')$, where $\deg(P)=d_1, \deg(Q)=d_2, \deg(P')=d_1', \text{and} \deg(Q')=d_2'$; that is, we record only the degrees rather than the polynomials themselves.
	 \footnotesize
	 \[\begin{tikzcd}
	 	{(r \cdot 12r)} & {r \cdot(6r \cdot 6r+6r \cdot 6r)} \\
	 	& {r \cdot \bigg((3r \cdot 3r+3r \cdot 3r) \cdot (3r \cdot 3r+3r \cdot 3r)+(3r \cdot 3r+3r \cdot 3r) \cdot (3r \cdot 3r+3r \cdot 3r)\bigg)}
	 	\arrow[from=1-1, to=1-2]
	 	\arrow["{\sigma(13)=6\,\, \text{is even}}", from=1-2, to=2-2]
	 \end{tikzcd}\]
	 \normalsize
	 Let $p=19$. Then $\sigma(p)=9$, which is not prime. We consider $F_2 \cdot G_2$, where $F_2$ is degree $r$ and $G_2$ is of degree $18r$. 
	 \footnotesize
	\[\begin{tikzcd}
		{(r \cdot 18r)} \\
		{r \cdot(9r \cdot 9r+9r \cdot 9r)} \\
		\begin{array}{c} r \cdot \bigg(\Big(3r \cdot (3-1)3r+3r \cdot (3-1)3r\Big) \cdot \Big((3r \cdot (3-1)3r+3r \cdot (3-1)3r\Big)+\\\Big((3r \cdot (3-1)3r+3r \cdot (3-1)3r\Big) \cdot \Big((3r \cdot (3-1)3r+3r \cdot (3-1)3r\Big)\bigg) \end{array} \\
		{r \cdot \bigg(\Big(3r \cdot (3r \cdot 3r+3r \cdot 3r)+3r \cdot (3r \cdot 3r+3r \cdot 3r)\Big) \cdot \Big(\cdots\Big)+\Big(\cdots\Big)\cdot\Big(\cdots\Big)\bigg)}
		\arrow[from=1-1, to=2-1]
		\arrow["{\sigma(19)=9\,\, \text{is odd; choose the smallest prime dividing} \,\sigma(19) }", from=2-1, to=3-1]
		\arrow["{\sigma(3)=\frac{3-1}{2}=1}", from=3-1, to=4-1]
	\end{tikzcd}\]
	\normalsize
	We describe the recursive method for an arbitrary prime $p$. Thus, $F_2$ is of degree $r$ and $G_2$ is of degree $(p-1)r$. If $\sigma(p)$ is prime, then by Corollary \ref{oddcase}, we have $M_p=(p-1)m_{\sigma(p)}^2$. If $\sigma(p)$ is not prime, we describe a method of writing $F_2G_2$ as a sum of products of homogeneous polynomials such that the degrees of these polynomials are of the form $q' \cdot r$, where $q'< p$ is a prime (possibly depending on the term). This provides a recursive step to compute the size of matrix factorization $M_p$. 
	
	Let $p'$ be the smallest prime that divides $\sigma(p)$, and we have $\sigma(p)=p' \cdot \alpha$ for some $\alpha \in \bN$.
	\footnotesize
	\[\begin{tikzcd}
		{(r \cdot (p-1)r)} \\
		{r \cdot\big(\sigma(p)r \cdot \sigma(p)r+\sigma(p)r \cdot \sigma(p)r \big)} \\
		\begin{array}{c} r \cdot \bigg(\Big(\alpha r \cdot (p'-1) \alpha r+\alpha r \cdot (p'-1) \alpha r\Big) \cdot \Big( \cdots\Big)+\\\Big(\cdots \Big) \cdot \Big(\cdots\Big)\bigg) \end{array} \\
		{r \cdot \bigg(\Big(\alpha r \cdot \big(\sigma(p')\alpha r \cdot \sigma(p') \alpha r+\sigma(p')\alpha r \cdot \sigma(p') \alpha r \big)+\alpha r \cdot \big(\cdots \big)\Big) \cdot \Big(\cdots\Big)+\Big(\cdots\Big)\cdot\Big(\cdots\Big)\bigg)}
		\arrow[from=1-1, to=2-1]
		\arrow["{\text{Smallest prime} \,\, p' \,\,\text{dividing}\, \sigma(p)}", from=2-1, to=3-1]
		\arrow[from=3-1, to=4-1]
	\end{tikzcd}\]
	 \normalsize

If $\alpha$ is a prime and $\sigma(p')=1$, then we can express $M_p$ in terms of $M_{\alpha}$ for some prime $\alpha < p$. This yields a recursive estimate. If either $\alpha$ is not prime or $\sigma(p') \ne 1$, we choose the smallest prime $p_1$ that divides $\alpha$ and the smallest prime $p_2$ that divides $\sigma(p') \cdot \alpha$, and repeat steps $3$ and $4$ in the above diagram. We provide a few additional examples to illustrate these steps.

Let $p=23$. Then $p$ is an admissible prime, and $M_p$ is estimated by Corollary \ref{oddcase}. 

Let $p=29$. Then $G_2$ is of degree $28r$. We have $\sigma(29)=14$. Then $p'=2$ is the  smallest prime dividing $\sigma(29)$, and we have $\alpha=7$ is also a prime. We have $p'-1=1$. Thus, the process terminates at step $3$, where every polynomial in the decomposition has degree $7 \cdot r$. Hence, $M_{29}$ can be computed in terms of $m_7$, and $m_7$ is given by Proposition \ref{easycase}.

\[\begin{tikzcd}
	{(r \cdot 28 r)} \\
	{r \cdot(14r \cdot 14r+14r \cdot 14r)} \\
	\begin{array}{c} r \cdot \bigg(\Big(7 r \cdot 7 r+7r \cdot7 r\Big) \cdot \Big( \cdots\Big)+\\\Big(\cdots \Big) \cdot \Big(\cdots\Big)\bigg) \end{array}
	\arrow[from=1-1, to=2-1]
	\arrow["{\text{Smallest prime} \,\, p'=2 \,\,\text{dividing}\, \sigma(p)}", from=2-1, to=3-1]
\end{tikzcd}\]

Let $p=31$. Then $G_2$ is of degree $30r$. We have $\sigma(31)=15$. Then $p'=3$ is the smallest prime dividing $\sigma(31)$, and we have $\alpha=5$ is also a prime. We have $\sigma(p')=\sigma(3)=1$. Thus, the process terminates at step $4$, as shown below.
\[\begin{tikzcd}
	{r \cdot \bigg(\Big( 5r \cdot (5 r \cdot  5r+ 5r \cdot  5r)+5r \cdot (\cdots)\Big) \cdot \Big(\cdots\Big)+\Big(\cdots\Big)\cdot\Big(\cdots\Big)\bigg)}
\end{tikzcd}\]

 Let $p=37$. Then $G_2$ is of degree $36r$. We have $\sigma(37)=18$. Then $p'=2$ is the  smallest prime dividing $\sigma(37)$. We have $\alpha=9$, which is not prime. We have $p'-1=1$. Thus, we choose the smallest prime $p_1=3$ that divides $9$. We write $\alpha=p_1 \cdot \beta$, and hence $\beta=3$. Step $5$ is $\alpha  r=(\beta r \cdot  (p_1-1) \beta r+ \beta r \cdot  (p_1-1) \beta r)$. The next step is $(p_1-1) \beta r=(\sigma(p_1)  \beta r \cdot \sigma(p_1)  \beta r+\sigma(p_1)  \beta r \cdot \sigma(p_1)  \beta r)$.
 
 \[\begin{tikzcd}
 	{(r \cdot 36 r)} \\
 	{r \cdot(18r \cdot 18r+18r \cdot 18r)} \\
 	\begin{array}{c} r \cdot \bigg(\Big(9 r \cdot 9 r+9r \cdot9 r\Big) \cdot \Big( \cdots\Big)+\\\Big(\cdots \Big) \cdot \Big(\cdots\Big)\bigg) \end{array} \\
 	\begin{array}{c} r \cdot \bigg(\Big((3r \cdot 6r+3r \cdot 6r) \cdot ()+ () \cdot ()\Big) \cdot \Big( \cdots\Big)+\\\Big(\cdots \Big) \cdot \Big(\cdots\Big)\bigg) \end{array} \\
 	\begin{array}{c} r \cdot \bigg(\Big((3r \cdot (3r \cdot 3r+3r \cdot 3r)+3r \cdot (3r \cdot 3r+3r \cdot 3r)) \cdot ()+ () \cdot ()\Big) \cdot \Big( \cdots\Big)+\\\Big(\cdots \Big) \cdot \Big(\cdots\Big)\bigg) \end{array}
 	\arrow[from=1-1, to=2-1]
 	\arrow["{\text{Smallest prime} \,\, p'=2 \,\,\text{dividing}\, \sigma(37)}", from=2-1, to=3-1]
 	\arrow["{p_1=3 \,\,\text{dividing} \,\, \alpha=9}", from=3-1, to=4-1]
 	\arrow["{\sigma(p_1)=1,\,\,\beta=3}", from=4-1, to=5-1]
 \end{tikzcd}\]

\subsubsection{Main theorem of this subsection}
 Let $p$ be a prime, not necessarily admissible. We recall that $M_{p}$ denotes the size of the matrix factorization, as in \eqref{newmfdefn}, of the homogeneous polynomial $F_2G_2$, where the entries of the matrices are homogeneous polynomials of degree $r$ (see Notation \ref{notamf}).

 When $d \cdot k$ is odd, we obtain the following estimate for the rank of an \ul bundle on generic cyclic coverings $X$ of $\bP^2$. This estimate is a consequence of the methods developed in this article.

\begin{theorem}
	\label{dkodd}
		Let $\pi\,:\, X \,\longrightarrow\, \bP^2$ be a generic cyclic covering of degree $d$ with branch divisor a curve of degree $d \cdot k$ such that $d \cdot k$ is odd. Let $p$ be the smallest prime that divides $d \cdot k$. Then $X$ will support an \ul bundle of rank $d \cdot M_p$.
\end{theorem}
\begin{proof}
	The branch locus $B$ of $\pi$ is given by a degree $d \cdot k$ polynomial $F$ which can be written as $F=F_{1}G_{1}+F_{2}G_{2}$. Here $F_1$ is of degree $k$, $F_{2}$ is of degree $r=\frac{d \cdot k}{p}$, and $G_{2}$ is of degree $(p-1)r$. By Proposition \ref{smoothnesstrans}, for a generic cyclic covering $X$ we can assume $F_1$ is smooth, and $F_1$ intersects $F_2 \cdot G_2$ transversally. We consider the following Cartesian diagram
	\begin{center}
		\begin{tikzcd}
			{{C}} && {{X}} \\
			{D} && {\mathbb{P}^2}
			\arrow[hook, from=2-1, to=2-3]
			\arrow["\pi'", from=1-1, to=2-1]
			\arrow["\pi"', from=1-3, to=2-3]
			\arrow[hook, from=1-1, to=1-3]
		\end{tikzcd}
	\end{center}
	where $D$ is a smooth curve of degree $k$ defined by $F_{1}$.

	Since $F_1$ intersects $F_2G_2$ transversally, the branch divisor of $\pi'$ is the intersection points $Z(F_2G_2) \cap D$, which are all distinct. Thus, the cyclic covering $C$ is also a smooth curve. We consider the Veronese embedding 
	\begin{equation}
		\begin{tikzcd}
			{\mathbb{P}^2} && {\mathbb{P}^N}
			\arrow["{\vert \mathcal{O}_{\mathbb{P}^2}(r)\vert}", hook, from=1-1, to=1-3]
		\end{tikzcd}
	\end{equation}
	The degree $p \cdot r$ polynomial $F_2G_2$ decomposes as a sum of products of degree $r$ homogeneous polynomials. We can rewrite the polynomial $F_2G_2$ as a degree $p$ homogeneous polynomial by replacing the degree $r$ homogeneous polynomials by linear polynomials using the Veronese coordinates of ${\bP^{N}}$. We denote this degree $p$ homogeneous polynomial in Veronese coordinates of ${\bP^{N}}$ by $F'$.
	
	Composing with the closed immersion $D \longhookrightarrow \bP^2$, we consider the following diagram
	\begin{center}
		\[\begin{tikzcd}
			& {Y=Z(t^p-F')\subseteq \bP^{N+1}} \\
			D & {\bP^{N}}
			\arrow["f", from=1-2, to=2-2]
			\arrow[hook, "i"', from=2-1, to=2-2]
		\end{tikzcd}\]
	\end{center}
	Let $C'$ be the following Cartesian product
	\begin{center}
		\begin{tikzcd}
			{{C'}} && {{Y}} \\
			{D} && {\mathbb{P}^{N}}
			\arrow[hook, "{i}", from=2-1, to=2-3]
			\arrow["f'", from=1-1, to=2-1]
			\arrow["f"', from=1-3, to=2-3]
			\arrow[hook, "j", from=1-1, to=1-3]
		\end{tikzcd}
	\end{center}
	Since the branch locus $Z(F') \subseteq {\mathbb{P}^{N}}$ of the cyclic covering $f$ pulls back to the branch locus $Zero(F_2G_2) \cap D$ under the map $i$, from the uniqueness of cyclic covering of $D$ for a fixed branch locus, we get $C=C'$ and $f'=\pi'$. 
	
	The polynomial $F_2G_2$ has a matrix factorization of size $M_p$ with entries of the matrices that are degree $r$ polynomials. An iterative formula for $M_p$ is provided in Corollary \ref{oddcase} for an admissible prime $p$ (see also \S 5.1.1 for the procedure to construct a matrix factorization for an arbitrary prime $p$). Hence, the degree $p$ polynomial $F'$ will have a linear matrix factorization of size $M_p$ with entries of the matrices that are linear polynomials in Veronese coordinates of  ${\mathbb{P}^{N}}$.
	 Using Remark \ref{keymfremark}, the polynomial $t^{p}-F'$ admits a linear matrix factorization of size $p \cdot M_{p}$. Thus, the cyclic covering $Y$ supports an \ul sheaf $E'$ of rank $M_{p}$ by Proposition \ref{keythm}. 
	The pullback $j^* E'$ will be a relatively \ul sheaf on $C$ by the base change theorem \cite[\href{https://stacks.math.columbia.edu/tag/02KG}{Tag 02KG}]{stacks-project}. 
	Since $C$ and $D$ are smooth, the pullback $j^* E'$ is a vector bundle on $C$ by Proposition \ref{intermediateprop}. Thus by Remark \ref{extensionthm}, we will get an \ul bundle of rank $d \cdot M_{p}$ on $X$.
\end{proof}

As a consequence, we get the following special case.
\begin{cor}
	\label{d=3}
Let $\pi\,:\, X \,\longrightarrow\, \bP^2$ be a generic smooth cyclic covering of degree $3$. Then there exists either a rank $3$ or a rank $6$ \ul vector bundle on $X$. 
\end{cor}
\begin{proof}
	Since the degree of the covering is $d\,=\,3$, the degree of the branch divisor is $d \cdot k$ for some $k \in \bN$. There are two cases depending on whether $k$ is even or odd. If $k$ is even, then $d \cdot k$ is even. Hence, by Theorem \ref{evenparitythm}, there exists a \ul bundle of rank $d=3$.
	
	Now we consider that $k$ is odd.  Hence, the prime $3$ is the smallest that divides $d \cdot k$. From Corollary \ref{oddcase}, we have $M_{3}=2$. Thus, by Theorem \ref{dkodd}, there exists an \ul bundle on $X$ of rank $d \,\cdot\, M_{3}\,=\,6$. 
\end{proof}

\section{ Some applications}\label{six}
In this section, we prove the following Corollary \ref{conic}, Corollary \ref{fermat}, Corollary \ref{newone}, and Example \ref{ellipticcase} as applications of the methods developed in \S 3, \S 4, and \S 5. The first two corollaries provide simpler proofs of earlier results obtained in \cite{parameswaran2021ulrich} using different methods. Corollary \ref{newone} provides new examples of cyclic coverings of $\bP^2$ of degree $d$ that admit Ulrich line bundles.
	\begin{cor}\label{conic}\cite[Corollary 1.2]{parameswaran2021ulrich}
	Let $\pi : X \longrightarrow \mathbb{P}^{2}$ be a degree $2$ covering such that the base curve is a smooth conic, then $X$ will admit an {U}lrich line bundle.
\end{cor}
\begin{proof}

Any non-degenerate conic is projectively equivalent to the smooth conic $C_{0} = \text{Zero}(y^{2} + x \cdot z)$ in $\mathbb{P}^{2}$. Thus, from the uniqueness of cyclic covering for a given branch, $X$ is defined by the polynomial $g = t^{2} - y^{2} - x \cdot z$ in $\mathbb{P}^{3}$. By Proposition \ref{keythm} to produce an \ul line bundle on $X$, it is enough to show that $g$ has a matrix factorization of size $2$.
	\begin{center}
	Let $A =  \begin{bmatrix}
		y & x \\
		z & -y 
	\end{bmatrix}  $, then $A^{2} =  \begin{bmatrix}
		y^{2}+x \cdot z & 0 \\
		0 & y^{2}+x \cdot z 
	\end{bmatrix}  = (y^{2}+x \cdot  z) \cdot \text{id}_{{2 \times2}}$.
\end{center} Thus we will have $$(t \cdot \text{id}_{{2 \times2}} + A)(t \cdot \text{id}_{{2 \times2}} - A) = t^{2}\cdot \text{id}_{{2 \times2}} - A^{2} = g \cdot \text{id}_{{2 \times2}}$$
	Let $\alpha_1$ be the matrix $(t \cdot \text{id}_{{2 \times2}} + A)$, and $\alpha_2$ be the matrix $(t \cdot \text{id}_{{2 \times2}} - A)$. Then $\alpha_1 \cdot \alpha_2=g \cdot \id$ is the desired matrix factorization. Then $G_1=\coker(\alpha_1)$, and $G_2=\coker(\alpha_2)$ are two \ul line bundles on $X$ given by Proposition \ref{keythm}.  
\end{proof}
\begin{cor}\cite[Theorem 1.6]{parameswaran2021ulrich}\label{fermat}
	Let $\pi \,:\, X \,\longrightarrow\, \mathbb{P}^{2}$ be a degree $2$ covering such that the branched curve is a Fermat curve of degree $2s$, then $X$ will admit an {U}lrich line bundle.
\end{cor}
\begin{proof}
	A Fermat curve of degree $2s$ is projectively equivalent to $\text{Zero}(x^{2s} - y^{2s} -z^{2s})$ in $\mathbb{P}^{2}$. Let $f: \mathbb{P}^{2} \longhookrightarrow \mathbb{P}^{N}$ be the Veronese embedding with respect to the linear system $|\mathcal{O}_{\mathbb{P}^{2}}(s)|$, where $N=\text{dim}\,H^0(\bP^2,\, \mathcal{O}_{\mathbb{P}^{2}}(s))-1$. Let  $z_{1}, z_{2}, \cdots, z_{N+1}$ be the Veronese coordinates of $\mathbb{P}^{N}$, where $z_1=x^s$, $z_2=y^s$, and $z_3=z^s$. We consider the degree $2$ cyclic covering $\tilde{\pi}: \tilde{X} = \text{Zero}(t^{2}-z_1^{2}+z_2^{2}+z_{3}^{2}) \subseteq \bP^{N+1}\longrightarrow \mathbb{P}^{N}$ over the branch locus $Zero(z_1^2-z_2^2-z_3^2) \subseteq \bP^N$. 
	\begin{center}
		Let $A =  \begin{bmatrix}
			z_1 & (z_2+iz_3) \\
			-(z_2-iz_3) & -z_1 
		\end{bmatrix}  $, then $A^{2} =  \begin{bmatrix}
			z_1^{2}-(z_2+iz_3) \cdot (z_2-iz_3) & 0 \\
			0 & z_1^{2}-(z_2+iz_3) \cdot (z_2-iz_3) 
		\end{bmatrix}  = (z_1^{2}-z_2^2-z_3^3) \cdot \text{id}_{{2 \times2}}$.
	\end{center} 
Then $\alpha_1=t \cdot \id_{2 \times 2}+A$, and $\alpha_2=t \cdot \id_{2 \times 2}-A$ gives the matrix factorization $\alpha_1 \cdot \alpha_2=(t^{2}-z_1^{2}+z_2^{2}+z_{3}^{2}) \cdot \id$ of size $2$. Let $G_i$ be the $\coker(\alpha_i)$ for $i=1, 2$ (see Definition \eqref{projdim}). Then by Theorem \ref{veroneseargumenttheorem} $(G_i)_{|X}$ is an \ul line bundle on $X$.    
\end{proof}

We generalize Corollary \ref{fermat} for a Fermat curve of degree $dk$, thereby producing new examples of cyclic coverings of $\bP^2$ of degree $d$ that admit Ulrich line bundles.

\begin{cor} \label{newone}
	Let $\pi\,:\,X \longrightarrow \bP^2$ be a degree $d$ cyclic covering branched over a Fermat curve of degree $dk$. Then $X$ supports an \ul line bundle.
\end{cor}

\begin{proof}
	Since the branch divisor $B$ is a Fermat curve in $\bP^2$ of degree $dk$, it is given by $B\,=\,Z(x^{dk}+y^{dk}-z^{dk})$. We consider the Veronese embedding $\bP^2 \longhookrightarrow \bP^M$ defined by the linear system $|\cO_{\bP^2}(k)|$, where $M\,=\,\dim H^0(\bP^2, \cO_{\bP^2}(k))-1$. Let $z_1, z_2, \cdots, z_{M+1}$ be the Veronose coordinates of $\bP^M$. Then $z_1, z_2, \cdots, z_{M+1}$ are given by the degree $k$ monomials in $\bC[x, y, z]$. In particular, we assume that $z_1=x^k, z_2=y^k$, and $z_3=z^k$. We consider the divisor $B'\,=\,Z(z_1^d+z_2^d-z_3^d) \subseteq \bP^M$ with the property $B' \cap \bP^2\,=\,B$. We define the hypersurface $\widetilde{Z}\,=\,Z(t^d-z_1^d-z_2^d+z_3^d) \subseteq \bP^{M+1}$. Then projection from the point $(0,0,\cdots, 1) \in \bP^{M+1}$ defines a cyclic covering $\widetilde{\pi}\,:\,\widetilde{Z} \longrightarrow \bP^M$ of degree $d$ branched over $B'$. Similar to the proof of Theorem \ref{veroneseargumenttheorem}, we obtain the following Cartesian diagram
	\begin{equation}
		\begin{tikzcd}
			{\text{X}} && {\widetilde{Z}} \\
			{\mathbb{P}^2} && {\mathbb{P}^M}
			\arrow["{|\mathcal{O}_{\mathbb{P}^2}(k)|}", hook, from=2-1, to=2-3]
			\arrow["\pi", from=1-1, to=2-1]
			\arrow["{\widetilde{\pi}}"', from=1-3, to=2-3]
			\arrow[hook, from=1-1, to=1-3]
		\end{tikzcd}
	\end{equation}
	
	Let $\omega$ be a primitive $d^{\text{th}}$ root of unity. Then we can write 
	$$t^d-z_1^d\,=\,(t-z_1)(t-\omega z_1)(t-\omega^2 z_1) \cdots \cdots (t-\omega^{d-1} z_1),$$
	and
$$z_3^d-z_2^d\,=\,(z_3-z_2)(z_3-\omega z_2)(z_3-\omega^2 z_2) \cdots \cdots (z_3-\omega^{d-1} z_2).$$
Thus, both $(t^d-z_1^d)$ and $(z_3^d-z_2^d)$ admit  matrix factorizations of size $1$.  Hence, by Remark \ref{keymfremark}, $(t^d-z_1^d-z_2^d+z_3^d)$ admits a matrix factorization of size $d \cdot 1 \cdot 1=d$,
$$Q_1 \cdot Q_2 \cdots Q_d\,=\, (t^d-z_1^d-z_2^d+z_3^d) \cdot \id_{d \times d},$$
where the $Q_i$ are $d \times d$ matrices whose entries are linear homogeneous polynomials in $\bC[t, z_1, z_2, z_3]$. Hence, by the proof of Theorem \ref{veroneseargumenttheorem}, there exist Ulrich sheaves $E_i\,=\,\coker(Q_i)$ (see Definition \ref{projdim}) of rank $1$ on $\widetilde{Z}$, and the restrictions $E_i \vert_{X}$ are Ulrich line bundles on $X$ for all $i=1, 2, \cdots, d$. 
\end{proof}

The following provides an example of a rank $3$ Ulrich bundle on a smooth cubic surface in $\bP^3$, given as a degree $3$ cyclic covering of $\bP^2$ branched over an elliptic curve.

\begin{example}
	\label{ellipticcase}
	Let $\pi\,:\,X \,\longrightarrow\, \bP^2$ be a degree $3$ cyclic covering such that the branch locus is an elliptic curve. Thus, in this case $d=3$ and $k=1$. By Corollary \ref{d=3}, there exists a rank $6$ \ul bundle on a generic degree $3$ cyclic covering $X$. In this particular case of $k=1$, we will construct a rank $3$ \ul bundle as follows. We are aware that there also exist rank $1$, and rank $2$ \ul bundles on $X$ by \cite[Example 3.5]{casanellas2012stable}. The Legendre form of the branch locus $F$ is given by $F = y^{2}z + x(x-z)(x-\lambda z)$. Hence the covering $X$ is the smooth cubic hypersurface $Zero(t^3-F) \subseteq \mathbb{P}^{3}$. 
	
	We will construct an explicit matrix factorization of $t^3-F$ of size $9 $.  Taking $\alpha_1 =  \begin{bmatrix}
		-y & 0 & x \\
		\frac{(x-z)}{2} & -\frac{y}{2} & 0 \\
		0 & x-\lambda z & z 
	\end{bmatrix}  $, $\alpha_2 =  \begin{bmatrix}
		-y & 0 & 2x \\
		(x-z) & -2z & 0 \\
		0 & x-\lambda z & y 
	\end{bmatrix}  $, $\alpha_3 =  \begin{bmatrix}
		z & 0 & x \\
		(x-z) & y & 0 \\
		0 & x-\lambda z & y 
	\end{bmatrix}  $, we can check that $\alpha_1 \cdot \alpha_2 \cdot \alpha_3 = F \cdot \id_{3 \times 3}$. By Remark \ref{keymfremark}, $(t^{3} - F)$ will have a matrix factorization of size $3 \cdot 1 \cdot 3\,=\,9$. Thus $X$ will support an {U}lrich bundle of rank $3$ by Proposition \ref{keythm}.

We can also see this using the method of \S \ref{four}. We take a projective line $\bP^1 \subseteq \bP^2$ intersecting the branch elliptic curve at three distinct points. We consider the following fiber product diagram
\begin{center}
	\begin{tikzcd}
		{{C}} && {{X}} \\
		{\bP^1} && {\mathbb{P}^{2}}
		\arrow["i", from=2-1, to=2-3]
		\arrow["\pi'", from=1-1, to=2-1]
		\arrow["\pi"', from=1-3, to=2-3]
		\arrow["i'", from=1-1, to=1-3]
	\end{tikzcd}
\end{center}

By the Riemann-Hurwitz formula, we have $$2g(C)-2=3(2g(\bP^1)-2)+(3-1)\cdot 3=-6+6=0.$$ Hence, we get $g(C)=1$. Thus $C$ is a smooth curve of genus $1$. Since the dimension of $\text{Pic}^3(C)\,=\,1$, almost all the line bundles of degree $3$ can not be written as $(\pi')^* \cO_{\bP^1}(1)$. Let $\cL$ be any such line bundle of degree $3$. By Riemann-Roch we have $\text{dim}\, H^0(X,\, \cL)=3$. We have $\pi_*\cL=\cM_1 \oplus \cM_2 \oplus \cM_3$. We can prove that $\sum_{i=1}^{3}\text{deg}\,(\cM_i)\,=\,0$ as follows. Let $\text{deg}\,(\cM_i)\, >\,0$ for some $i$. Since $\cL \otimes (\pi')^* \cO_{\bP^1}(-1) \ncong \cO_{C}$, we will have $H^0(C,\, \cL \otimes (\pi')^* \cO_{\bP^1}(-1))\,=\,0$. On the other hand, using the projection formula, we have $H^0(C,\, \cL \otimes (\pi')^* \cO_{\bP^1}(-1)) \,\cong\, H^0(\bP^1,\, (\cM_1 \oplus \cM_2 \oplus \cM_3) \otimes \cO_{\bP^1}(-1)) \,\ne\, 0$ since $\text{deg}\,(\cM_i(-1)) \,\ge\, 0$ for some $i$, which is a contradiction. Thus, we get $\pi_* \cL \cong \cO_{\bP^1}^{\oplus 3}$. Hence, $X$ admits an \ul bundle of rank $3$ by the Remark \ref{extensionthm}.
 
\end{example}

We therefore pose the following question.
\begin{quest} \label{parityquest}
	In Theorem \ref{evenparitythm} we have proved the existence of a rank $d$ \ul bundle on a generic degree $d$ cyclic covering of $\bP^2$ when the degree of the branch $d \cdot k$ is even. Is it true that there exists an  \ul bundle of rank $\le d$  on every smooth cyclic covering $X \longrightarrow \bP^2$ of degree $d$, irrespective of the parity of $d \cdot k$?  
\end{quest}

\section*{Acknowledgement}  The second-named author of this paper was supported by a Postdoctoral Fellowship of TIFR, Mumbai, while the first-named author was a faculty there during the initial preparation phase of this work. The second-named author acknowledges support from a Visiting Scientist position at ISI Bangalore. The second-named author is currently supported by an Institute Postdoctoral Fellowship at IISER Pune. The authors thank the referee for a meticulous reading of the article and for many valuable suggestions that have greatly improved it.\

\bibliographystyle{alphaurl}
\bibliography{Ulrich_bundles_cyclic_covering_arxiv.bib}

\end{document}